\documentclass[a4paper, 11pt, oneside, reqno, notitlepage]{amsart}
\usepackage{amsmath,amssymb,amsthm,graphicx,mathrsfs,bbm,url}
\usepackage{amsthm}
\usepackage{wrapfig}
\usepackage{enumitem}
\usepackage{mathtools}
\usepackage[usenames,dvipsnames]{xcolor}
\usepackage[colorlinks=true,linkcolor=Red,citecolor=Green]{hyperref}
\usepackage[super]{nth}
\usepackage[open, openlevel=2, depth=3, atend]{bookmark}
\hypersetup{pdfstartview=XYZ}
\usepackage[font=footnotesize]{caption}
\usepackage{subcaption}
\usepackage{caption}
\usepackage{tikz}
\usepackage{setspace}
\usetikzlibrary{decorations.pathreplacing}
\usetikzlibrary{arrows}
\usetikzlibrary{shapes.misc}
\usetikzlibrary{shapes.symbols}
\usetikzlibrary{patterns}
\usepackage[normalem]{ulem}

\usepackage{graphicx}
\usepackage{cite}
\usepackage[colorinlistoftodos]{todonotes}
\usepackage{bbm}
\usepackage[top=3.35cm, bottom=3.35cm, left=2.5cm, right=2.5cm, headsep=0.2in]{geometry}

\usepackage{epstopdf}
 
\usepackage{hyperref}

\newcommand{\mc}{\mathcal}

\newcommand{\comp}{\mathrm{comp}}

\newcommand{\Gr}{\mathrm{Gr}}
\newcommand{\GL}{\mathrm{GL}}

\DeclareMathOperator{\Tr}{Tr}

\DeclareMathOperator{\im}{Im}
\DeclareMathOperator{\re}{Re}

\DeclareMathOperator{\supp}{supp}
\DeclareMathOperator{\vol}{vol}

\DeclareMathOperator{\id}{Id}

\DeclareMathOperator{\Lie}{\mathcal{L}}

\DeclareMathOperator{\e}{\mathbf{e}}

\theoremstyle{plain}
\newtheorem{theorem}{Theorem}[section]
\newtheorem{definition}[theorem]{Definition}
\newtheorem{lemma}[theorem]{Lemma}
\newtheorem{remark}[theorem]{Remark}
\newtheorem{proposition}[theorem]{Proposition}

\numberwithin{equation}{section}

\begin{document}

\title{Calder\'on problem for systems via complex parallel transport}
\author[M. Ceki\'{c}]{Mihajlo Ceki\'{c}}
\date{\today}
\address{Institut f\"ur Mathematik, Universit\"at Z\"urich, Winterthurerstrasse 190, CH-8057 Z\"urich, Switzerland}
\email{mihajlo.cekic@math.uzh.ch}

\begin{abstract}
    We consider the Calder\'on problem for systems with unknown zeroth and first order terms, and improve on previously known results. More precisely, let $(M, g)$ be a compact Riemannian manifold with boundary, let $A$ be a connection matrix on $E = M \times \mathbb{C}^r$ and let $Q$ be a matrix potential. Let $\Lambda_{A, Q}$ be the \emph{Dirichlet-to-Neumann map} of the associated connection Laplacian with a potential. Under the assumption that $(M, g)$ is isometrically contained in the interior of $(\mathbb{R}^2 \times M_0, c(e \oplus g_0))$, where $(M_0, g_0)$ is an \emph{arbitrary} compact Riemannian manifold with boundary, $e$ is the Euclidean metric on $\mathbb{R}^2$, and $c > 0$,  we show that $\Lambda_{A, Q}$ uniquely determines $(A, Q)$ up to natural gauge invariances. Moreover, we introduce new concepts of \emph{complex ray transform} and \emph{complex parallel transport} problem, and study their fundamental properties and relations to the Calder\'on problem.
\end{abstract}

\maketitle

\section{Introduction} 
One of the most prominent and well-studied inverse problems is the Calder\'on problem (see \cite{uhlmann-09}), where one is interested in reconstructing an unknown elliptic PDE coefficient from boundary measurements. The foundational work of Sylvester-Uhlmann \cite{sylvester-uhlmann-87} was based on probing the domain of interest with \emph{complex geometric optics (CGO)} solutions, and after much effort this approach was generalised to transversally anisotropic geometries in \cite{dos-santos-ferreira-keng-salo-uhlmann-09}. The latter approach was applied to the setting of systems (or connections) in \cite{cekic-17} with limited success and the main aim of the present paper is to push this further to higher rank. This generalises the earlier work of Eskin \cite{eskin-01} to anisotropic geometries.

\subsection{Calder\'on problem for systems} Let $(M, g)$ be a smooth, connected, compact, Riemannian manifold with boundary. We consider the trivial rank $r \in \mathbb{Z}_{\geq 1}$ vector bundle $E = M \times \mathbb{C}^r$ equipped with the standard Hermitian inner product in its fibres. A \emph{connection} (resp. \emph{unitary connection}) $A$ on $E$ can be identified with a smoothly varying (resp. skew-Hermitian) matrix of $1$-forms, giving rise to a \emph{covariant derivative} $d_A := d + A$, mapping the space of vector valued sections $C^\infty(M, \mathbb{C}^r)$ to the space of vector valued $1$-forms $C^\infty(M, T^*M \otimes \mathbb{C}^r)$. Consider also a smooth matrix potential $Q \in C^\infty(M, \mathbb{C}^{r \times r})$. Define the operator
\[
    \Lie_{A, Q} := d_{-A^*}^* d_A + Q,
\]
where $A^*$ denotes the Hermitian adjoint of $A$, and $d_{-A^*}^*$ denotes the formal adjoint of $d_{-A^*}$ with respect to the natural $L^2$ structures (see Proposition \ref{prop:laplace-local-coord} for a local coordinate expression). Assuming that $0$ is not a Dirichlet eigenvalue of $\Lie_{A, Q}$ on $M$, it is possible to define the \emph{Dirichlet-to-Neumann (DN) map} $\Lambda_{A, Q}: C^\infty(\partial M, \mathbb{C}^r) \to C^\infty(\partial M, \mathbb{C}^r)$ as $\Lambda_{A, Q} f := \iota_\nu d_A u|_{\partial M}$, where $u$ is the unique smooth solution to the Dirichlet problem
\[
    \Lie_{A, Q}u = 0, \quad u|_{\partial M} = f,
\]
and $\nu$ is the outer boundary normal to $\partial M$, and $\iota_\nu$ denotes contraction with $\nu$. If $G \in C^\infty(M, \GL_r(\mathbb{C}))$, where $\GL_r(\mathbb{C})$ is the group of invertible complex $r \times r$ matrices, it is straightforward to show (see Proposition \ref{prop:DN-invariance}) that $G^{-1} \Lie_{A, Q} G = \Lie_{G^*A, G^*Q}$, where
\[
    G^*A := G^{-1} dG + G^{-1} A G, \quad G^*Q := G^{-1} Q G.
\]
Note that if $G$ takes values in the unitary group $\mathrm{U}(r)$ and $A$ is unitary, then $G^*A$ is also a unitary connection. We refer to $G$ as a \emph{gauge transformation}. Assuming $G|_{\partial M} = \id$, it follows that $\Lambda_{G^*A, G^* Q} = \Lambda_{A, Q}$, and the \emph{Calder\'on problem for systems (connections)} becomes: can one uniquely determine the pair $(A, Q)$ up to such gauge transformations $G$ from the DN map $\Lambda_{A, Q}$? This is partially answered in the following theorem. We denote by $e$ the Euclidean metric on $\mathbb{R}^2$ and by \emph{isometrically contained} we mean that there is an embedding of codimension zero which is an isometry onto image.

\begin{theorem}\label{thm:main-theorem-gaussian-beams}
    Let $(M, g)$ be a smooth connected compact Riemannian manifold with boundary, isometrically contained in the interior of $(\mathbb{R}^2 \times M_0, c(e \oplus g_0))$, where $(M_0, g_0)$ is a smooth compact Riemannian manifold with boundary and dimension $\dim M_0 \geq 1$, and $c \in C^\infty(\mathbb{R}^2 \times M_0, \mathbb{R}_{> 0})$ is a conformal factor. Let $r \in \mathbb{Z}_{\geq 1}$, let $A_1$ and $A_2$ be two smooth connections on $E = M \times \mathbb{C}^r$, and let $Q_1, Q_2 \in C^\infty(M, \mathbb{C}^{r \times r})$ be two smooth matrix potentials. Assume that the Dirichlet-to-Neumann maps agree, that is, $\Lambda_{A_1, Q_1} = \Lambda_{A_2, Q_2}$ on $C^\infty(\partial M, \mathbb{C}^r)$. Then, there exists a gauge transformation $G \in C^\infty(M, \GL_r(\mathbb{C}))$, with $G|_{\partial M} = \id$, such that
	\[
		G^{-1} dG + G^{-1} A_1 G = A_2, \quad G^{-1} Q_1 G = Q_2.
	\]	
    Moreover, if $A_1$ and $A_2$ are assumed to be unitary connections, then we may require that $G$ takes values in $\mathrm{U}(r)$.
\end{theorem}

The case $r = 1$ of the above theorem is also a consequence of \cite[Theorem 1.3]{salo-17} and \cite[Theorem 1.5]{cekic-17}. The former proves injectivity of the $X$-ray transform on product manifolds and although it is stated only for functions, its proof should carry over to the case of $1$-forms. The novelty of our theorem is for $r \geq 2$, where we give the \emph{most general geometric assumptions} available in the literature. 

Let us explain a bit more the context. Firstly, Eskin \cite{eskin-01} proved uniqueness for Euclidean geometries, i.e. when $(M, g)$ is isometrically contained in a Euclidean space, and our result may be seen as a generalisation to the \emph{anisotropic setting}. Moreover, in author's previous work \cite{cekic-17}, the more general geometry $(\mathbb{R} \times M_0, c(e \oplus g_0))$ was considered and the Calder\'on problem (for unitary connections and zero potential) was reduced to proving uniqueness for a certain complex parallel transport problem, see \cite[Question 1.6]{cekic-17}. This uniqueness is still not known, and an ancillary objective of this article is to shed some light on that question, see \S \ref{ssection:complex-parallel-transport} below for more details. In this context, allowing the extra Euclidean direction in $(\mathbb{R}^2 \times M_0, c(e \oplus g_0))$ is natural, and another novelty is that \emph{we do not require any additional assumption} on the $(M_0, g_0)$ factor. Indeed, the typical assumption present in the literature for the Calder\'on problem is that of injectivity of the $X$-ray transform on functions and $1$-forms on this factor, see \cite{dos-santos-ferreira-keng-salo-uhlmann-09, cekic-17, krupchyk-uhlmann-18}. Finally, our result allows to consider non-unitary connections and non-skew-Hermitian potentials simultaneously, as opposed to e.g. \cite{cekic-17, st-amant-24}. A more detailed overview of relevant literature is given below in \S \ref{ssec:previous-work}.

\subsection{Complex ray transform and complex parallel transport.}\label{ssection:complex-parallel-transport} The secondary objective of this paper is to introduce and study basic properties of a new type of integral transform, called the \emph{complex ray transform}, and its \emph{non-linear} analogue for connections, the \emph{complex parallel transport} (see Section \ref{sec:complex-ray-transform}). 

Roughly speaking, given a Riemannian manifold with boundary $(M, g)$, and a function (or a tensor) $f$ with compact support in the interior of $\mathbb{R} \times M$, for a given point $x \in \partial M$ and a non-trapped inward pointing $v \in T_xM$ generating a geodesic $\gamma$ in $M$, the complex ray transform is defined as the boundary value of the solution to the Del Bar equation along $\mathbb{R} \times \gamma$ with $f$ as a source term, see Figure \ref{fig:complex-ray}. This object can be twisted with a connection on the vector bundle $E = M \times \mathbb{C}^r$, see Definition \ref{def:complex-ray} below for more details. The complex parallel transport is defined similarly by taking an invertible matrix solution of a homogeneous Del Bar equation along $\mathbb{R} \times \gamma$ and restricting it to a suitable boundary curve in $\mathbb{R} \times \gamma$ as before, see Definition \ref{def:complex-parallel-transport-0} below. (However, there are several competing possible definitions for the notion of complex parallel transport, see Remark \ref{remark}.) The real counterparts of these complex objects, the (attenuated) $X$-ray transform and the parallel transport problem (or holonomy inverse problem) are very well-studied in literature, see e.g. \cite{Paternain-Salo-Uhlmann-12, Guillarmou-Paternain-Salo-Uhlmann-16, Cekic-Lefeuvre-21-1} and references therein.

\begin{center}
\begin{figure}[htbp!]
\includegraphics[scale=0.75]{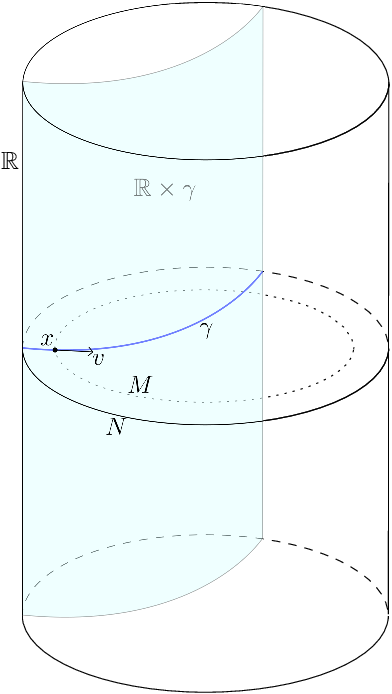}
\caption{\label{fig:complex-ray} In shaded blue is the bicharacteristic ray $\mathbb{R} \times \gamma$ generated by the geodesic $\gamma$ starting at $x \in \partial M$ and with speed $v$; $N$ is the extension of $M$ constructed in \S \ref{ssec:extensions}.}
\end{figure}
\end{center} 

Except for intrinsic interest, the complex parallel transport appears naturally when studying the Calder\'on problem in the \emph{conformally transversally anisotropic (CTA)} geometry $(\mathbb{R} \times M, c(e \oplus g))$, where $c > 0$ (see \cite[Question 1.6]{cekic-17}). Our aim here is to put these complex objects inside a rigorous mathematical framework, prove some basic properties, and provide some perspectives. Moreover, after constructing suitable CGO solutions, Theorem \ref{thm:main-theorem-gaussian-beams} above can be thought of as inverting a suitable complex parallel transport problem with \emph{additional data} coming from the fact that $\mathbb{R}^2 \times M_0$ has a CTA structure in any of the $\mathbb{R}^2$ directions. 

The main results of Section \ref{sec:complex-ray-transform} are that, essentially, \emph{when $r = 1$} both the linear and the non-linear questions reduce to their real counterparts on the transversal manifold $(M, g)$, see Propositions \ref{prop:injectivity} and \ref{prop:dimension-two}. The arguments are similar to the ones in \cite{dos-santos-ferreira-keng-salo-uhlmann-09, cekic-17} used in the context of the Calder\'on problem. \emph{For $r > 1$} the injectivity of the complex parallel transport problem and complex ray transform is \emph{completely unknown} and is left as a subject of future research. However, distinctively to the real case, the complex parallel transport problem is not injective for $\dim M = 1$, see Proposition \ref{prop:counterexample}.

Let us also mention that the complex ray transform bears formal similarities with the \emph{light ray transform}, which in the setting of Lorentzian geometry asks whether a function can be recovered from its integrals along lightlike geodesics (see \cite{lassas-oksanen-stefanov-uhlmann-20}). This similarity goes further by noting that the light ray transform naturally appears in the Lorentzian Calder\'on problem (see for instance \cite{feizmohammadi-ilmavirta-kian-oksanen-21}); on the level of PDEs this is explained because for the hyperbolic and elliptic problems the suitable CGO solutions concentrate on light rays and bicharacteristic planes (products $\mathbb{R} \times \gamma$ as above), respectively. 

\subsection{Previous work on Calder\'on type problems for systems.}\label{ssec:previous-work} In the CTA setting, as mentioned above, we reduced the problem (when $(M_0, g_0)$ is \emph{simple}, the connections are unitary, and the potential is zero) to what we call below the complex parallel transport problem, see \cite[Question 1.6]{cekic-17}. When $r = 1$ we uniquely determined the connection assuming only partial data and that $(M_0, g_0)$ has injective $X$-ray transform; see also \cite[Theorem 4]{dos-santos-ferreira-keng-salo-uhlmann-09}. Eskin \cite{eskin-01} proved uniqueness for \emph{any} $r \geq 1$ in the Euclidean setting, i.e.\ $(M, g) \Subset (\mathbb{R}^n, e)$. Using different methods, uniqueness in the case of \emph{Yang-Mills} connections and arbitrary geometry was shown in \cite{cekic-20}. In the analytic case, uniqueness was obtained by Gabdurakhmanov-Kokarev \cite{gabdurakhmanov-kokarev-21}. Last but not the least, Albin-Guillarmou-Tzou-Uhlmann \cite{albin-guillarmou-tzou-uhlmann-13} solve the problem fully if $\dim M = 2$ and the connection is unitary. St-Amant \cite{st-amant-24} recently showed uniqueness (in the case of zero potential and unitary connections) for a large but fixed frequency under an assumption of injectivity for the parallel transport problem.

The case $r = 1$ (\emph{magnetic Laplacian}) was thoroughly studied in the literature and we mention only a few results in this direction: in the CTA setting Krupchyk-Uhlmann \cite{krupchyk-uhlmann-18} showed uniqueness assuming the connection and the potential are only $L^\infty$ regular; partial data problem with rough coefficients was recently studied by Selim-Yan \cite{selim-yan-23}; uniqueness under essentially minimal regularity was obtained in the Euclidean setting by Haberman \cite{haberman-18}.

The setting of Yang-Mills-Dirac equations in the analytic category was considered by Valero \cite{valero-22}. The related problem for the dynamic Schr\"odinger equation was tackled by Tetlow \cite{tetlow-22}. In the hyperbolic setting, the recovery of time-independent connection and potential was carried out by Kurylev-Oksanen-Paternain \cite{kurylev-oksanen-paternain-18}. For non-local equations, zeroth and first order terms were uniquely recovered (there is \emph{no} gauge invariance) by the author with Lin and R\"uland \cite{cekic-lin-rueland-20}. In the context of non-linear equations, a connection was uniquely determined in Minkowski space up to gauge equivalence from the source-to-solution map of the twisted cubic wave equation by Chen-Lassas-Oksanen-Paternain \cite{chen-lassas-oksanen-paternain-22}.

\subsection{Proof ideas.} The idea is to use the extra freedom (compared to the CTA setting) in choosing the Euclidean direction in $\mathbb{R}^2$ and to construct suitable CGO solutions probing the manifold. A \emph{bicharacteristic leaf} is a totally geodesic surface of the form $\mathbb{R}\mu \times \gamma \subset (\mathbb{R}^2 \times M_0, e \oplus g_0)$, where $0 \neq \mu \in \mathbb{R}^2$ and $\gamma$ is orthogonal to $\mu$. We parameterise bicharacteristic leaves by the \emph{$2$-frame bundle} $L$, consisting of non-zero pairs of vectors $(\mu, \nu)$ with same norm, where $\mu \in \mathbb{R}^2$ and $\nu \perp \mu$. The leaves foliate $\mathbb{R}^2 \times M_0$. This construction is introduced in detail and the general construction of CGO solutions based on solving $\overline{\partial}$-equations (along the leaves) is recalled in \S \ref{ssec:LCW}, \S \ref{ssec:CGO}, and start of \S \ref{ssec:CGO-del-bar}.

For one fixed bicharacteristic leaf, a family of CGOs depending on a semiclassical parameter $h > 0$ concentrating on that leaf as $h \to 0$ was constructed in \cite{cekic-17}. The novelty in this paper is to smoothly glue the solutions to the leafwise $\overline{\partial}$-equation together and to reduce the theorem to a transport (tomography) problem on $L$. Typically, in previous works there was little need to smoothly parameterise the solutions to the transport equation, as for $r = 1$ it is possible to directly obtain a reduction to the $X$-ray transform, and the construction of invertible solutions to the $\overline{\partial}$-equation is trivial since we may take as ansatz the exponential of a solution to an inhomogeneous $\overline{\partial}$-equation (solved by applying Cauchy's operator). 

The reduction to a transport problem on $L$ is done in two steps: firstly, in Lemma \ref{lemma:parametric-del-bar} we construct invertible solutions to the matrix $\overline{\partial}$-equation along the bicharacteristic leaves parametrised locally from the boundary (we use a result of Nakamura-Uhlmann \cite{nakamura-uhlmann-02} in the process). Secondly, in Theorem \ref{thm:reduction-to-complex-x-ray-connection} we use these local solutions to the leafwise $\overline{\partial}$-equation to construct CGOs. Plugging the CGOs into the integral (Alessandrini) identity and taking the $h \to 0$ limit (Step 1), shows that we can modify these leafwise solutions so that by uniqueness they glue together nicely on the boundary (Step 2), and that they extend to the interior since any leaf intersects the boundary (Steps 3 and 4). 

Thus we obtain a smooth invertible matrix function $G$ on $L$ satisfying the complex transport problem
\begin{equation}\label{eq:transport-intro}
    \mathbb{X}(\mu + i\nu) G(\mu + i\nu) + A_1(x_1, x_2, x)(\mu + i\nu) G(\mu + i\nu) - G(\mu + i\nu) A_2(x_1, x_2, x)(\mu + i\nu) = 0,
\end{equation}
with $G \equiv \id$ in the exterior of $M$, where $\mathbb{X}$ is a \emph{complexified} geodesic vector field (i.e. a complex differential operator on $L$ defined by differentiating along the parallel transport of $(\mu, \nu)$ along geodesics in directions $\mu$ and $\nu$), and $\mu + i\nu$ represents the $2$-frame $(\mu, \nu) \in L$. Inspired by Eskin \cite{eskin-01} we consider sufficiently large holomorphic families of $2$-frames depending on $t\in \mathbb{C}\setminus \{0\}$ and differentiate \eqref{eq:transport-intro} using the complex derivative $\partial_{\bar{t}}$: we get $\partial_{\bar{t}} G \equiv 0$ and by Liouville's theorem that $G$ is constant in these families and eventually independent of the $2$-frame variable (see Lemma \ref{lemma:connection-uniqueness}). This shows $G^*A_1 = A_2$, and the case of potentials is dealt with similarly.

\subsection{Further remarks and perspectives.} It is natural to ask whether the proof of Theorem \ref{thm:main-theorem-gaussian-beams} can work with an arbitrary Riemannian metric on $\mathbb{R}^2$ (which by the existence of isothermal coordinates is always conformal to $e$). However, obstructions for the existence of \emph{limiting Carleman weights (LCW)} (see \S \ref{ssec:LCW} below) and so for suitable CGO solutions on products of surfaces obtained in \cite{angulo-ardoy-faraco-guijarro-ruiz-17} say that the metric on $\mathbb{R}^2$ has to be given by a surface of revolution in which case there is essentially only \emph{one} LCW and we cannot apply our proof.

As mentioned above, we believe that the arguments of this paper could be applied to the setting of the parallel transport problem in the Lorentzian setting \cite{feizmohammadi-ilmavirta-kian-oksanen-21}, if the spacial part has enough Euclidean directions. For the complex parallel transport problem it would be interesting to study the injectivity in the simplest non-trivial case: when $M$ is a planar disk, and $B = 0$ (see Definition \ref{def:complex-parallel-transport}). Next, it could be worthwhile to compare the situation when $M$ is a simple surface with the recent results of Bohr-Paternain on the transport twistor space \cite{Bohr-Paternain-23} and possibly re-interpret the injectivity of the complex parallel transport problem. Understanding the boundary value problems for the (twisted) $\overline{\partial}$-equation might shed some light on complex parallel transport problem; for this, it could be useful to consult the work of Kenig \cite{Kenig-23} (and the Wiener-Masani theorem). Finally, the statements in this paper are made for $C^\infty$-regularity; it could be interesting to lower the regularity to $L^\infty$ or to some $C^k$ where $k$ is explicit (see \cite{krupchyk-uhlmann-18}). This is left for future investigation.

\subsection{Organisation of the paper.} In Section \ref{sec:preliminaries}, we lay out some preliminaries: in \S \ref{ssec:connection-laplacian} about basic properties of $\Lie_{A, Q}$ and symmetries of $\Lambda_{A, Q}$, and of the integral (Alessandrini) identity, in \S \ref{ssec:boundary-determination} about boundary determination, that is, the determination of the full jet of $A$ and $Q$ on $\partial M$ from $\Lambda_{A, Q}$, in \S \ref{ssec:del-bar} we state a solvability result for the Del Bar equation, in \S \ref{ssec:extensions} we discuss extensions of Riemannian manifolds, and finally in \S \ref{ssec:topological-reduction} certain topological arguments about embeddings in $\mathbb{R}^2 \times M_0$. In Section \ref{sec:CGO-del-bar} we state Theorem \ref{thm:main-theorem-simple} which is the same as Theorem \ref{thm:main-theorem-gaussian-beams} with the additional assumption that $(M_0, g_0)$ is \emph{simple} (this makes the construction of CGOs more explicit and we recommend the reader to first go through the proof of this theorem). Here we also introduce the new basic objects, the bundle of $2$-frames and the complexified geodesic vector field, and then go on to reduce the question of injectivity to certain (complex) transport problems, see Theorems \ref{thm:reduction-to-complex-x-ray-connection} and \ref{thm:reduction-to-complex-x-ray-potential}. In Section \ref{sec:simple-transversal-geometry} we show injectivity for the obtained transport problems, again in the setting of simple transversal geometry and show Theorem \ref{thm:main-theorem-simple}. In Section \ref{sec:gaussian-beams} we show Theorem \ref{thm:main-theorem-gaussian-beams}; we explain the main differences with the proof of Theorem \ref{thm:main-theorem-simple} and in particular, we use the CGO construction via Gaussian Beams obtained in \cite{cekic-17}. Finally, in Section \ref{sec:complex-ray-transform} we introduce the \emph{complex ray transform} and the \emph{complex parallel transport} problem, and study their basic injectivity properties. 
\medskip

{\bf Acknowledgements.} During the course of writing this project I received funding from an Ambizione grant (project number 201806) from the Swiss National Science Foundation and was supported by the Max Planck Institute in Bonn. I am grateful to Katya Krupchyk, François Monard, Lauri Oksanen, and Mikko Salo for helpful discussions. I would also like to warmly thank two anonymous referees for their valuable comments and questions which improved the article.

\section{Preliminaries}\label{sec:preliminaries}

In this section we spell out a few preliminary technical results. We will use the standing assumptions that $(M, g)$ is a smooth compact Riemannian manifold with boundary with $n := \dim M$, $r \in \mathbb{Z}_{\geq 1}$, $A, A_1, A_2$ are connections on $E = M \times \mathbb{C}^r$, and $Q, Q_1, Q_2 \in C^\infty(M, \mathbb{C}^{r \times r})$ are matrix valued potentials. Unless otherwise stated, we assume that $0$ is not a Dirichlet eigenvalue of $\Lie_{A_i, Q_i}$ for $i = 1, 2$ or $\Lie_{A, Q}$ (where these operators are defined in \S \ref{ssec:connection-laplacian} below). We will use the summation convention.

\subsection{Connection Laplacian}\label{ssec:connection-laplacian} Here we derive elementary facts about the connection Laplacian. The following statement is certainly well-known but we could not locate a precise reference. Write $|g| = \det g$ for the determinant of the Riemannian metric in local coordinates.

\begin{proposition}\label{prop:laplace-local-coord}
    In local coordinates $(x_i)_{i = 1}^n$, where $A = A_i dx_i$ we may write for any local section $u$ of $M \times \mathbb{C}^r$
    \begin{align*}
        \Lie_{A}u = d_{-A^*}^* d_A u &= -|g|^{-\frac{1}{2}} (\partial_{x_i} + A_i) (|g|^{\frac{1}{2}} g^{ij} (\partial_{x_j} + A_j))u\\
        &= -\Delta_g u  - 2g^{ij} A_i \partial_{x_j}u + (d^*A)u - g^{ij} A_i (Au)_j. 
    \end{align*}
\end{proposition}
\begin{proof}
    Let $(\e_i)_{i = 1}^r$ be the standard basis of $\mathbb{C}^r$; write $A_i \e_j = A_{i, kj} \e_k$, where $A_{i, kj}$ is the $kj$-th entry of $A_i$ and consider a section $u = u_i \e_i$. Then we have
    \begin{align}\label{eq:d_A}
        d_A u = du_i \otimes \e_i + u_i dx_j \otimes A_j \e_i = dx_j \otimes (\partial_{x_j} + A_j) u.
    \end{align}
    Next, we compute in the coordinate patch $U$, given smooth functions $f_{\ell k}$ for $\ell = 1, \dotsc, n$, $k = 1, \dotsc, r$, and $v := f_{\ell k} dx_{\ell} \otimes \e_k$
    \begin{align*}
        \int_U \langle{d_A u, v}\rangle\, d\vol_g &= \int_U \langle{dx_j \otimes (\partial_{x_j} + A_j) u, f_{\ell k} dx_{\ell} \otimes \e_k}\rangle\, d\vol_g\\
        &= \int_U g^{j \ell} \bar{f}_{\ell k} (\partial_{x_j} u_k + \langle{A_j \e_i, \e_k}\rangle u_i) |g|^{\frac{1}{2}}\, dx\\
        &= -\int_U |g|^{-\frac{1}{2}} \partial_{x_j}(\bar{f}_{\ell k} g^{j\ell} |g|^{\frac{1}{2}}) u_k\, d\vol_g + \int_U g^{j\ell} A_{j, ki} \bar{f}_{\ell k} u_i\, d\vol_g\\
        &= - \int_U \langle{u, \e_k \otimes |g|^{-\frac{1}{2}} \partial_{x_j} (f_{\ell k} g^{j\ell} |g|^{\frac{1}{2}}) - \bar{A}_{j, ki} \e_i g^{j \ell} f_{\ell k}}\rangle\, d\vol_g\\
        &= - \int_U \langle{u, |g|^{-\frac{1}{2}} (\partial_{x_j} - A_j^*)(|g|^{\frac{1}{2}} g^{j \ell} f_{\ell k}) \e_k}\rangle\, d\vol_g,
    \end{align*}
    where we use \eqref{eq:d_A} in the first equality, $d\vol_g = |g|^{\frac{1}{2}}\,dx$ in the second one, integration by parts in the third equality, and we denoted by $A_j^*$ the adjoint of a matrix. Therefore the formal adjoint of $d_A$ is given by
    \begin{equation}\label{eq:d_A^*}
        d_A^*v = -|g|^{-\frac{1}{2}} (\partial_{x_j} - A_j^*)(|g|^{\frac{1}{2}} g^{j \ell} f_{\ell k}) \e_k = -|g|^{-\frac{1}{2}} (\partial_{x_j} - A_j^*)(|g|^{\frac{1}{2}} \langle{dx_j, v}\rangle). 
    \end{equation}
    Combining \eqref{eq:d_A} and \eqref{eq:d_A^*} immediately proves the first formula.    

    For the second one, we observe
    \begin{equation}\label{eq:d_A^*-alternative}
        d_A^*v = -|g|^{-\frac{1}{2}} \partial_{x_j}(|g|^{\frac{1}{2}} g^{j \ell} f_{\ell k}) \otimes \e_k + A_j^* g^{j \ell} f_{\ell} = d^*v + A_j^*g^{j\ell} v_\ell,
    \end{equation}
    where in the first equality we used \eqref{eq:d_A^*} and wrote $v_{\ell} := f_{\ell k} \e_k$ for the $dx_\ell$ component of $v$, and in the last equality we recognised the formula for the codifferential $d^*$. Combining \eqref{eq:d_A} and \eqref{eq:d_A^*} shows the second formula and completes the proof.
\end{proof}

In particular we deduce the gauge invariance as follows.

\begin{proposition}\label{prop:DN-invariance}
    Let $F \in C^\infty(M, \GL_r(\mathbb{C}))$. Then we have
    \[
        F^{-1} d_{-A^*} F = d_{-F^*A}^*, \quad F^{-1} \Lie_{A, Q} F = \Lie_{F^*A, F^*Q}.
    \]
    Moreover, if $F|_{\partial M} = \id$, then $\Lambda_{F^*A, F^*Q} = \Lambda_{A, Q}$.
\end{proposition}
\begin{proof}
    By definition of pullback of a connection we have 
    \begin{equation}\label{eq:gauge-covariant-formula}
        F^{-1} d_A F = d_{F^*A}
    \end{equation}
    we will use local coordinates notation as in Proposition \ref{prop:laplace-local-coord}. Write $F \e_i = F_{ji} \e_j$ for the components of $F$. Moreover, given $v = f_{\ell k} dx_{\ell} \otimes \e_k$ we compute
    \begin{equation}\label{eq:gauge-covariant-adjoint-formula}
        F^{-1} d_{-A^*}^* F v = -|g|^{-\frac{1}{2}} F^{-1}(\partial_{x_j} + A_j)F (|g|^{\frac{1}{2}} g^{j \ell} f_{\ell k})\e_k = d_{-F^*A}^* v,
    \end{equation}
    where in the first equality we used \eqref{eq:d_A^*} and in the last equality we used $F^{-1} d_A F = d_{F^*A}$ and \eqref{eq:d_A^*} again. Therefore
    \[
        F^{-1}\Lie_{A, Q} F = F^{-1} d_{-A^*}^* F F^{-1} d_A F + F^*Q = \Lie_{F^*A, F^*Q},
    \]
    where in the second equality we used \eqref{eq:gauge-covariant-formula} and \eqref{eq:gauge-covariant-adjoint-formula}. The equality of the DN maps follows from the second formula.
\end{proof}

Next we will derive a twisted version of Green's identity.

\begin{proposition}\label{prop:green}
    For all $u \in C^\infty(M, T^*M \otimes \mathbb{C}^r)$ and $v \in C^\infty(M, \mathbb{C}^r)$, we have
    \[
        \langle{d_A^*u, v}\rangle_{L^2(M, \mathbb{C}^r)} - \langle{u, d_A v}\rangle_{L^2(M, \mathbb{C}^r)} = -\langle{\iota_\nu u, v}\rangle_{L^2(\partial M, \mathbb{C}^r)}.
    \]
\end{proposition}
\begin{proof}
    In local coordinates, we compute 
    \begin{align*}
        \langle{d_A^*u, v}\rangle_{L^2} - \langle{u, d_A v}\rangle_{L^2} &= 
        \langle{d^* u, v}\rangle_{L^2} - \langle{u, d v}\rangle_{L^2} + \underbrace{\langle{A_j^* g^{j \ell} u_\ell, v}\rangle_{L^2} - \langle{u, dx_i \otimes A_i v}\rangle_{L^2}}_{T := },
    \end{align*}
    where in the second line we used \eqref{eq:d_A^*-alternative} and \eqref{eq:d_A}; $u_\ell$ denotes the $dx_{\ell}$ component of $u$. Then we have pointwise
    \[
        \langle{A_j^*g^{j \ell} u_\ell, v}\rangle = \langle{dx_j, dx_\ell}\rangle \langle{A_j^* u_\ell, v}\rangle = \langle{u, Av}\rangle,
    \]
    showing that $T \equiv 0$. The main identity is then a consequence of the usual Green's formula, see \cite[Equation (2.1)]{cekic-17}.
\end{proof}

We now compute the adjoint of the DN map.

\begin{proposition}\label{prop:DN-adjoint-formula}
    We have $\Lie_{A, Q} ^* = \Lie_{-A^*, Q^*}$ and $0$ is not a Dirichlet eigenvalue of $\Lie_{-A^*, Q^*}$. Moreover, for any $f, h \in C^\infty(\partial M, \mathbb{C}^{r \times r})$, we have
    \[
        \langle{\Lambda_{A, Q}f, h}\rangle_{L^2} = \langle{f, \Lambda_{-A^*, Q^*}h}\rangle_{L^2} = \langle{d_A u, d_{-A^*} w}\rangle_{L^2} + \langle{Qu, w}\rangle_{L^2},
    \]
    where $u$ and $w$ are unique solutions to
    \[
        \Lie_{A, Q} u = 0,\,\, u|_{\partial M} = f, \quad \Lie_{-A^*, Q^*} w = 0,\,\, w|_{\partial M} = h.
    \]
\end{proposition}
\begin{proof}
    The adjoint formula follows by definition of $\Lie_{A, Q}$. In particular, since $0$ is not the Dirichlet eigenvalue of $\Lie_{A, Q}$ by assumption, and since the index of $\Lie_{A, Q}$ is zero (there is a continuous deformation to $-\Delta_g$), we conclude that $0$ is not a Dirichlet eigenvalue of $\Lie_{-A^*, Q^*}$ either, and $\Lambda_{-A^*, Q^*}$ is well-defined. By Proposition \ref{prop:green} we have
    \[
        \langle{d_{-A^*} d_A u, v}\rangle_{L^2} - \langle{d_A u, d_{-A^*} v}\rangle_{L^2} = -\langle{\iota_\nu d_A u, v}\rangle,
    \]
    where $u$ is as in the statement and $v$ can for the moment be any smooth section with $v|_{\partial M} = h$. This implies
    \[
        \langle{\Lambda_{A, Q} f, h}\rangle_{L^2} = \langle{d_A u, d_{-A^*} v}\rangle_{L^2} + \langle{Qu, v}\rangle_{L^2}.
    \]
    Similarly, for $u$ and $w$ as in the statement
    \[
        \langle{\Lambda_{-A^*, Q^*}h, f}\rangle_{L^2} = \langle{d_{-A^*}w, d_A u}\rangle_{L^2} + \langle{Q^*u, w}\rangle_{L^2}.
    \]
    Plugging in $v := w$ in the former formula and identifying with the complex conjugate of the latter one gives the required result.
\end{proof}

Note that the previous results with obvious extensions also hold for sections of $M \times \mathbb{C}^{r \times r}$ equipped with the standard Hermitian metric induced by trace, $\langle{M_1, M_2}\rangle = \Tr(M_1M_2^*)$; the covariant derivative is extended by $d_A U = dU + AU$ for matrix valued sections $U$. We reach the main result of this subsection; its version for \emph{unitary} connections was proved in \cite[Theorem 2.5]{cekic-17}. We will write $|A|^2 = g^{ij} A_i A_j$ (defined in local coordinates).

\begin{proposition}\label{prop:integral-identity}
    Let $F_1, F_2 \in C^\infty(\partial M, \mathbb{C}^{r \times r})$ and consider the unique solutions $U_1, U_2 \in C^\infty(M, \mathbb{C}^{r \times r})$ to 
		\[
			\Lie_{A_1, Q_1} U_1 = 0,\,\, U_1|_{\partial M} = F_1, \quad \Lie_{-A_2^*, Q_2^*} U_2 = 0,\,\, U_2|_{\partial M} = F_2.
		\]
	Then:
\begin{align*}
	    &\langle{(\Lambda_{A_1, Q_1} - \Lambda_{A_2, Q_2}) F_1, F_2}\rangle_{L^2(\partial M, \mathbb{C}^{r \times r})}\\ 
		&=  \langle{(Q_1 - Q_2 + |A_2|^2 - |A_1|^2) U_1, U_2}\rangle_{L^2(M, \mathbb{C}^{r \times r})} + \langle{U_1\cdot dU_2^* - dU_1\cdot U_2^*, A_1^* - A_2^*}\rangle_{L^2(M, \mathbb{C}^{r \times r})}.
\end{align*}	
\end{proposition}
\begin{proof}
    By Proposition \ref{prop:DN-adjoint-formula} we compute
    \begin{align*}
        &\langle{(\Lambda_{A_1, Q_1} - \Lambda_{A_2, Q_2}) F_1, F_2}\rangle_{L^2(\partial M, \mathbb{C}^{r \times r})}\\ 
        &= \langle{\Lambda_{A_1, Q_1} F_1, F_2}\rangle_{L^2(\partial M, \mathbb{C}^{r \times r})} - \langle{F_1, \Lambda_{-A_2^*, Q_2^*} F_2}\rangle_{L^2(\partial M, \mathbb{C}^{r \times r})}\\
        &= \langle{(Q_1 - Q_2) U_1, U_2}\rangle_{L^2(M, \mathbb{C}^{r \times r})} + \underbrace{\langle{d_{A_1}U_1, d_{-A_1^*} U_2}\rangle_{L^2(M, \mathbb{C}^{r \times r})} - \langle{d_{A_2} U_1, d_{-A_2^*} U_2}\rangle_{L^2(M, \mathbb{C}^{r \times r})}}_{T :=}.
    \end{align*}
    We are left to identify the term $T$ with the formula in the statement. Indeed, we have pointwise in local coordinates
    \[
        \langle{dU_1, -A_1^* U_2}\rangle = \langle{\partial_{x_i} U_1 \cdot dx_i, -A_{1j}^* U_2 \cdot dx_j}\rangle = -g^{ij} \Tr (\partial_{x_i} U_1 \cdot U_2^* A_{1j}) = \langle{dU_1 \cdot U_2^*, -A_1^*}\rangle,
    \]
    where we write $A_{i} = A_{ij} dx_j$ for $i = 1 ,2$, as well as
    \[
        \langle{A_1 U_1, dU_2}\rangle = g^{ij} \Tr(A_{1i}U_1 \cdot \partial_{x_j} U_2^*) = g^{ij} \Tr(U_1 \cdot\partial_{x_j}U_2^* \cdot A_{1i}) = \langle{U_1 \cdot dU_2^*, A_1^*}\rangle, 
    \]
    and also
    \[
        \langle{A_1 U_1, -A_1^* U_2}\rangle = -g^{ij} \Tr(A_{1i}U_1 \cdot U_2^* A_{1j}) = -g^{ij} \Tr(A_{1j} A_{1i} U_1 U_2^*) = -\langle{|A_1|^2 U_1, U_2}\rangle.
    \]
    Combining the three previous formulas, we obtain
    \begin{align*}
        T &= \langle{dU_1, (A_2^* - A_1^*) U_2}\rangle_{L^2} + \langle{(A_1 - A_2) U_1, dU_2}\rangle_{L^2} + \langle{A_1 U_1, -A_1^* U_2}\rangle_{L^2} - \langle{A_2 U_1, -A_2^* U_2}\rangle_{L^2}\\
        &= \langle{d U_1\cdot U_2^*, A_2^* - A_1^*}\rangle_{L^2} + \langle{U_1 \cdot dU_2^*, A_1^* - A_2^*}\rangle_{L^2} + \langle{(|A_2|^2 - |A_1|^2) U_1, U_2}\rangle_{L^2},
    \end{align*}
    which brings $T$ to the required form and completes the proof.
\end{proof}

\subsection{Boundary determination}\label{ssec:boundary-determination} The following result was proven in \cite[Theorem 3.4]{cekic-20} for \emph{unitary} connections, and we now give its version in the non-unitary case.

\begin{proposition}\label{proposition:boundary-determination}
    If $\Lambda_{A_1, Q_1} \equiv \Lambda_{A_2, Q_2}$, there exist $G_1, G_2 \in C^\infty(M, \GL_r(\mathbb{C}))$, with $G_i|_{\partial M} \equiv \id$, such that
	\[
		G_2^*A_2 - G_1^*A_1, \quad G_2^* Q_2 - G_1^*Q_1,
	\]
	have full jets equal to zero on $\partial M$. Moreover, there exists $\varepsilon > 0$ such that in boundary normal coordinates in the $\varepsilon > 0$ neighbourhood of $\partial M$, the normal components of $G_2^*A_2$ and $G_1^*A_1$ vanish; if $A_i$ are unitary we may also require that $G_i$ take values in unitary matrices, for $i = 1, 2$.
\end{proposition}
\begin{proof}
    In fact, the proof of \cite[Theorem 3.4]{cekic-20} applies also in this setting, by simply noticing that the form of the connection Laplacian in \cite[Equation (3.3)]{cekic-20} is precisely the same (algebraically) as the expression for $\Lie_{A, Q}$ derived in Proposition \ref{prop:laplace-local-coord}. Indeed, since the symbol expansion of the pseudodifferential operator $\Lambda_{A, Q}$ derived after (3.3) is derived from the coefficients of $\Lie_{A, Q}$, the result follows.
\end{proof}

In the following sections we will often work with a (codimension zero) embedding $M \Subset N^\circ \subset N$, and assume that $A_i$ and $Q_i$ are arbitrarily extended from $M$ to connections and matrix potentials on $N$, which are compactly supported on $N^\circ$, respectively for $i = 1, 2$. Moreover, if the DN maps $\Lambda_{A_1, Q_1}$ and $\Lambda_{A_2, Q_2}$ agree, using Proposition \ref{proposition:boundary-determination}, we will assume that the corresponding extensions are chosen such that they agree on the complement of $M$.

\subsection{The Del Bar equation}\label{ssec:del-bar} We will also record a useful PDE lemma about the Del Bar equation. For $\tau \in \mathbb{R}$, introduce the weighted $L^2$ norm by
\[
    \|u\|_{\tau}^2 := \int_{\mathbb{R}^2} (1 + |x|^2)^\tau |u(x)|^2\, dx.
\]
Then $L^2_\tau(\mathbb{R}^2)$ is the closure of $C_{\comp}^\infty(\mathbb{R}^2)$ with respect to $\|\bullet\|_\tau^2$ and defines a Hilbert space. We will identify $\mathbb{R}^2 \simeq \mathbb{C}$ and write $\partial_{\bar{z}} := \frac{1}{2}(\partial_{x_1} + i\partial_{x_2})$. The following was proved in \cite[Lemma 5.1]{nakamura-uhlmann-02}.

\begin{lemma}\label{lemma:del-bar-source}
    Let $A \in C_{\comp}^\infty(\mathbb{R}^2, \mathbb{C}^{r \times r})$. Then, there exist $C > 0$ and a negative constant $\tau \in \mathbb{R}$ with $|\tau|$ large enough, such that for any $f\in L^2_{\tau + 2}(\mathbb{R}^2, \mathbb{C}^{r \times r})$, the equation
    \[
        \partial_{\bar{z}} u(z) + A(z) u(z) = f(z), \quad z \in \mathbb{C},  
    \]
    admits a unique solution $u \in L^2_\tau(\mathbb{R}^2, \mathbb{C}^{r \times r})$ with the estimate
    \[
        \|u\|_\tau \leq C \|f\|_{\tau + 2}.
    \]
\end{lemma}

\subsection{Extensions.}\label{ssec:extensions} Assume now $(M, g)$ has strictly convex boundary. On the \emph{unit sphere bundle} $SM := \{(x, v) \in TM \mid g_x(v, v) = 1\}$, define the \emph{geodesic vector field} $X$ by
\[
    X(x, v) := \frac{d}{dt}\Big|_{t = 0} (\gamma(t), \dot{\gamma}(t)), \quad (x, v) \in SM, 
\]
where $\gamma(t)$ is the geodesic in $M$ defined by the initial condition $\gamma(0) = x$ and $\dot{\gamma}(0) = v$. By the results in \cite[Section 2]{Guillarmou-17-1}, there is an extension $(N, g)$ of $(M, g)$ (where the metric is denoted with the same letter), such that $M \Subset N^\circ$, with the following properties
\begin{itemize}
	\item $\partial N$ is strictly convex;
	\item there is a vector field $\widetilde{X}$ on $SN$, complete in $SN^\circ$, such that it agrees with the geodesic vector field $X$ on $SM$;
	\item every trajectory of $\widetilde{X}$ that leaves $SM$ never comes back.
\end{itemize}
Moreover, if $(M, g)$ is simple, i.e. the boundary of $\partial M$ is strictly convex (in the sense that the second fundamental form is strictly positive) and the exponential map is a diffeomorphism from its domain of definition for any point in the interior $M^\circ$, then since the simplicity property is $C^2$ open, there is a simple manifold $(M_e, g) \Subset (N^\circ, g)$ containing $M$ in its interior.

 \subsection{Topological and geometrical reductions.}\label{ssec:topological-reduction} 
 
    Let us state auxiliary topological lemmas that will be applied in the setting of the main theorem. Assume $M$ connected and embedded (with codimension zero) into $\mathbb{R}^2 \times M_0$ where $M_0$ has connected boundary. Call the connected component(s) of $\partial M$ lying in the connected component of $\mathbb{R}^2 \times \partial M_0$ inside $\mathbb{R}^2 \times M_0 \setminus M^\circ$ the \emph{outer boundary} of $M$; any other component of $\partial M$ will be called \emph{inner boundary}. Then we have:
	
	\begin{lemma}\label{lemma:topology-1}
		Assume $M_0$ is diffeomorphic to a ball. Then, the outer boundary is connected and bounds a compact connected codimension zero smooth submanifold $\widetilde{M} \subset \mathbb{R}^2 \times M_0$.
	\end{lemma}
		\begin{proof}
			Assume the outer boundary had (at least) two distinct connected components $B_1$ and $B_2$. By the Jordan-Brouwer separation theorem, each one of $B_i$ bounds a connected compact smooth manifold $D_i$ with $\partial D_i = B_i$, such that $\mathbb{R}^2 \times M_0 \setminus D_i$ is connected and unbounded, for $i = 1, 2$. Therefore, $D_1$ is in the exterior of $D_2$ and vice-versa, as otherwise this would contradict $B_1$ and $B_2$ being the outer boundary. But $D_1$ and $D_2$ being in each others complement contradicts the fact that $M$ is connected.
			
			The remaining claim follows again from the Jordan-Brouwer separation theorem.
		\end{proof}

  We remark that simple manifolds are necessarily diffeomorphic to a ball. We now give a slightly more general reduction.

  \begin{lemma}\label{lemma:topology-2}
      Assume $(M, g) \Subset (\mathbb{R}^2 \times M_0^\circ, e \oplus g_0)$ with $(M_0, g_0)$ arbitrary smooth, connected, and compact Riemannian manifold. Then there exists a smooth, connected, and compact $(\widetilde{M}_0, g_0)$ with strictly convex boundary diffeomorphic to a sphere extending $(M_0, g_0)$. Moreover, there exists a smooth, connected, compact, and codimension zero submanifold $\widetilde{M} \subset \mathbb{R}^2 \times \widetilde{M}_0$ containing $M$ such that $\partial \widetilde{M}$ is precisely the outer boundary.
  \end{lemma}
  \begin{proof}
      For the first claim, consider first the \emph{double} $(M_0', g_0)$ of $(M_0, g_0)$, that is, the manifold obtained by smoothly gluing $M_0$ with itself along the boundary, and extending $g_0$ smoothly in an arbitrary way. Then, delete a small disk $D$ from the complement of $(M_0, g_0)$ inside $(M_0', g_0)$ and in its place glue a cylinder diffeomorphic to $[0, 1] \times \mathbb{S}^{n - 1}$, with a metric $g_{\mathrm{cyl}}$ which is strictly convex metric at one of the boundary components. This completes the proof.

      To see the second statement, consider a connected component $B$ of the inner boundary of $\partial M$. Consider the connected component $\mc{U}$ of $\mathbb{R}^2 \times M_0 \setminus M^\circ$ which has $B$ as one of the components of the boundary. We claim that $\mc{U}$ is compact and all the components of $\partial \mc{U}$ are inner boundary. Indeed, if $\mc{U}$ is non-compact, then it is connected to $\mathbb{R}^2 \times \partial M_0$, which contradicts the fact that $B$ is inner boundary. Next, if there is a connected component $B'$ of $\partial \mc{U}$ which is outer boundary of $M$, this would again contradict the fact that $B$ is inner boundary.

      Therefore, there is a finite number of compact smooth disjoint manifolds with boundary on the inner boundary of $M$, denoted by $\mc{U}_1, \dotsc, \mc{U}_k$, such that $\widetilde{M} := M \cup \bigcup_{i = 1}^k \mc{U}_i$ satisfies the conditions of the lemma. This completes the proof.
  \end{proof}
  
\section{Complex Geometric Optics and the Del Bar equation}\label{sec:CGO-del-bar}

In this section we consider the easier setting of \emph{simple} (see \S \ref{ssec:extensions} for a definition) transversal manifolds. We construct the Complex Geometric Optics (CGO) solutions in this setting and as a consequence reduce the main problem to solving a certain complex transport problem. Before proving Theorem \ref{thm:main-theorem-gaussian-beams}, we would first like to prove:

\begin{theorem}\label{thm:main-theorem-simple}
	With the same assumptions as in Theorem \ref{thm:main-theorem-gaussian-beams}, with the exception that $(M_0, g_0)$ is now a simple manifold, the same conclusion holds.
\end{theorem}

\subsection{Limiting Carleman weights}\label{ssec:LCW} Recall that in general, $(N, g_N)$ admits a \emph{limiting Carleman weight} (LCW) if there is an open neighbourhood $(\mc{O}, g_{\mc{O}})$ of $(N, g_N)$ and $\varphi \in C^\infty(\mc{O})$ with non-vanishing gradient, such that the semiclassical Weyl principal symbol $p_\varphi$ of the semiclassical operator $P_\varphi := e^{-\frac{\varphi}{h}} (-h^2 \Delta_g) e^{\frac{\varphi}{h}}$ satisfies
\[
	\{\re p_\varphi, \im p_\varphi\} = 0 \quad \mathrm{when}\quad  p_\varphi = 0.
\]
If this condition is satisfied, then the characteristic set $p_\varphi^{-1}(0) \subset T^*M$ is foliated by $2$-dimensional leaves called \emph{bicharacteristic leaves}, which are tangent to the Hamiltonian vector fields $H_{\re p_{\varphi}}$ and $H_{\im p_\varphi}$. 

We elucidate this definition in the case $(N, g_N) \subset (\mathbb{R}, e) \times (N_0, g_{0N})$; then the coordinate function $\varphi := x_1$ is an LCW: indeed then 
\[
    P_{\varphi} = -h^2 \Delta_g - 1 - 2h\partial_{x_1}, 
\]
and one sees that $p_\varphi(x, \xi) = |\xi|^2 -1 - 2i \xi_1$ and so $p_\varphi^{-1}(0) = \mathbb{R} \times S^*N_0 \subset T^*(\mathbb{R} \times N_0)$, where $S^*N_0 \subset T^*N_0$ denotes the space of unit co-tangent vectors, that we freely identify via musical isomorphism with the unit tangent bundle $SN_0 \subset TN_0$. Then $H_{\re p_\varphi}$ is proportional to the geodesic vector field on $SN_0$ and $H_{\im p_\varphi}$ is proportional to $\partial_{x_1}$ on $\mathbb{R}$ (commuting vector fields), and the bicharacteristic leaf at $(x_1, x, v)$ is given by 
\[
    \{(x_1 + t, \gamma(s), \dot{\gamma}(s)) \mid t \in \mathbb{R},\, s \in (a, b)\} \subset \mathbb{R} \times SN_0,
\]
where $\gamma$ is the geodesic generated by $(x, v)$, and $(a, b)$ its maximal domain of existence. We will often identify this bicharacteristic leaf with its projection onto $\mathbb{R} \times N_0$.

\subsection{CGO construction}\label{ssec:CGO} We collect a few necessary facts about CGO solutions that will get used later; this construction was carried out in \cite[Section 4]{cekic-17} in the case of \emph{unitary} connections. The proof carries over to arbitrary connections and we outline it for completeness. In fact, in this subsection in order to simplify the notation we will work with the \emph{transversally anisotropic (TA)} assumption $(M, g) \Subset (\mathbb{R} \times M_0, e \oplus g_0)$, and we will write $x_1$ for the $\mathbb{R}$ coordinate and $x$ for the $M_0$ coordinate. Enlarge $(M_0, g_0)$ to a slightly larger simple manifold $(N_0, g_{0})$ containing $M_0$ in its interior; write $A$ for a smooth connection on $M \times \mathbb{C}^r$ and $Q \in C^\infty(M, \mathbb{C}^{r \times r})$ for a smooth matrix potential.
 
	Let $p \in \partial N_0$ and consider the polar coordinate system $(r, \theta)$ at $p$. We look for matrix solutions to $\mc{L}_{A, Q} U = 0$ of the form $U := U_h(x_1, r, \theta) := e^{-\frac{x_1 + ir}{h}}(a + R)$ where $h > 0$ will be a small positive (semiclassical) parameter. Setting $\rho := x_1 + ir$ we compute
	\begin{equation}\label{eq:old}
		e^{\frac{\rho}{h}} h^2 \mc{L}_{A, Q} e^{-\frac{\rho}{h}} a = h^2 \mc{L}_{A, Q}a - h\big((-\Delta_g \rho)a - 2g^{ij} A_i (\partial_{x_j}\rho) a - 2\langle{d\rho, da}\rangle\big),
	\end{equation}
	see for instance \cite[Lemma 4.1]{cekic-17}, where we used that $|dx_1 + i dr|^2 = 0$, and wrote $\langle{\bullet, \bullet}\rangle$ and $|\bullet|^2$ for the complex \emph{bilinear} extension of the Riemannian inner product and norm, respectively. Then using
	\begin{align*}
		(-\Delta_g)\rho &= -|g|^{-\frac{1}{2}}  (\partial_{x_1} + i \partial_r) |g|^{\frac{1}{2}},\\
		-2g^{ij} A_i \partial_{x_j} \rho &=  -2A(\partial_{x_1} + i \partial_r) = -2(A_1 + iA_r),\\
		\langle{d\rho, da}\rangle &= (\partial_{x_1} + i \partial_r)a,
	\end{align*}
	where $|g| = \det g$, $A_1 := A(\partial_{x_1})$, and $A_ r := A(\partial_r)$, and writing $\partial_{\bar{z}} := \frac{1}{2} (\partial_{x_1} + i \partial_r)$, we equate the second term of \eqref{eq:old} to zero:
	\[
		\partial_{\bar{z}}a + \frac{1}{2}(A_1 + iA_r) a + \frac{1}{2} |g|^{-\frac{1}{2}} \partial_{\bar{z}} |g|^{\frac{1}{2}} \cdot a = 0.
	\]
	Substituting $b := a |g|^{\frac{1}{4}}$ we obtain the equation
	\begin{equation}\label{eq:transport-1}
		\partial_{\bar{z}}b + \frac{1}{2} (A_1 + iA_r) b = 0.
	\end{equation}
	If instead of $\rho = x_1 + ir$ we chose $\rho = -x_1 + ir$ we would obtain the transport equation
	\begin{equation}\label{eq:transport-2}
		\partial_{z} b + \frac{1}{2} (A_1 - iA_r) b = 0,
	\end{equation}
	where $\partial_z := \frac{1}{2} (\partial_{x_1} - i \partial_r)$. Assuming we are able to solve \eqref{eq:transport-1} with a smooth $b$, to satisfy $\mc{L}_{A, Q}e^{-\frac{\rho}{h}}(a + R) = 0$ we are left to solve
	\[
		e^{\frac{\rho}{h}} \mc{L}_{A, Q} e^{-\frac{\rho}{h}} R = -\mc{L}_{A, Q} a \iff e^{\frac{x_1}{h}} \mc{L}_{A, Q} e^{-\frac{x_1}{h}} (e^{-\frac{ir}{h}}R)= -e^{-\frac{ir}{h}}\mc{L}_{A, Q} a.
	\]
	By \cite[Theorem 3.5]{cekic-17}, there is a solution $R = R_h$ to this equation satisfying
	\[
		\|R_h\|_{L^2} = \mc{O}(h), \quad \|R_h\|_{H^1} = \mc{O}(1), \quad \mathrm{as}\quad h \to 0.
	\]

\subsection{CGOs and the Del Bar equation}\label{ssec:CGO-del-bar}
Assume now $(M, g) \Subset (\mathbb{R}^2 \times M_0^\circ, e \oplus g_0)$. To parametrise bicharacteristic leaves introduced in \S \ref{ssec:LCW}, for the linear LCW given by the Euclidean inner product $\varphi(x_1, x_2, x) := (x_1, x_2) \cdot \mu$ where $0 \neq \mu \in \mathbb{R}^2$, we need one direction in the $\mathbb{R}^2$ direction and one orthogonal direction. More precisely, the space of such leaves is at a point $(x_1, x_2, x) \in \mathbb{R}^2 \times M_0$ modelled by 
\[
	\Gr_2(n) \setminus \Gr_2(n - 2) = \Gr_1(2) \times \Gr_1(n - 2) \cup \Gr_2(2),
\]
where by $\Gr_i(j)$ we denote the Grassmanian of $i$-planes in $\mathbb{R}^j$. The problem is that the latter factor is a point and so this space is \emph{not a manifold} and is not suitable for doing analysis. However, the fibre bundle of (unoriented) $2$-frames with one direction in $\mathbb{R}^2$ is a manifold:
\[
	\mathbb{F}_2 := \{(x_1, x_2, x, \mu, \nu) \mid \mu \in S_{(x_1, x_2)}\mathbb{R}^2, \,\, \nu \in S_{(x_1, x_2, x)}(\mathbb{R}^2 \times M_0),\,\, \mu \perp \nu\},
\]
Observe that $\mathbb{F}_2$ can be naturally identified with the manifold of \emph{complex light rays} 
\[
	L_1 :=  \{(x_1, x_2, x, \mu + i\nu) \in T_{\mathbb{C}} (\mathbb{R}^2 \times M_0) \mid |\mu + i\nu|^2 = 0,\,\, \mu \in S_{(x_1, x_2)} \mathbb{R}^2\}, 
\]
where $T_{\mathbb{C}}(\mathbb{R}^2 \times M_0)$ is the complexification of $T(\mathbb{R}^2 \times M_0)$ and $|\bullet|^2$ is the extension to $T_{\mathbb{C}}(\mathbb{R}^2 \times M_0)$ of Riemannian norm by complex bilinearity. Note that $|\mu + i\nu|^2 = 0$ is equivalent to $\mu$ and $\nu$ being perpendicular and of same norm. We will also need
\[
	L :=  \{(x_1, x_2, x, \mu + i\nu) \in T_{\mathbb{C}} (\mathbb{R}^2 \times M_0) \mid |\mu + i\nu|^2 = 0,\,\, \mu \neq 0\} \cong \mathbb{R}_{> 0} \times L_1.
\]

We are able to determine the topology of $L_1$ (and so of $L$).

\begin{lemma}
	We have $L_1$ is a smooth $\mathbb{S}^1 \times \mathbb{S}^{n -2}$-fibre bundle. If $(M_0, g_0)$ is simple, then $L_1$ is diffeomorphic to $\mathbb{R}^2 \times M_0 \times \mathbb{S}^1 \times \mathbb{S}^{n - 2}$.
\end{lemma}
\begin{proof}
	The smoothness part follows from definition; consider the circle bundle $K := \{(x_1, x_2, x, \mu) \in \mathbb{R}^2 \times M_0 \mid \mu \in S_{(x_1, x_2)}\mathbb{R}^2\} \cong \mathbb{R}^2 \times M_0 \times \mathbb{S}^1$. There is a projection of fibre bundles $L_1 \to K$ given by $(x_1, x_2, x, \mu + i\nu) \mapsto (x_1, x_2, x, \mu)$. In the case $(M_0, g_0)$ is simple, since $\mathbb{R}^2$ and $M_0$ are contractible (the latter by the definition of being simple, as the domain of the exponential map is star-shaped), we have that $K \cong \mathbb{R}^2 \times M_0 \times \mathbb{S}^1$ and $L_1 \cong \mathbb{R}^2 \times M_0 \times F$, where $F$ is a fibre to be determined and which is a fibre bundle $F \to \mathbb{S}^1$. The vector bundle over $\mathbb{S}^1 \subset \mathbb{R}^n$ consisting of vectors orthogonal to the points of $\mathbb{S}^1$ in $\mathbb{R}^n$ is diffeomorphic to $T\mathbb{S}^1 \oplus \mathbb{R}^{n - 2} \cong \mathbb{S}^1 \times \mathbb{R}^{n - 1}$. Therefore the bundle $F \to \mathbb{S}^1$ is the trivial $\mathbb{S}^{n - 2}$ bundle and the result follows. The explicit fibre bundle isomorphism is given by
	\[
		\mathbb{R}^2 \times M_0 \times \mathbb{S}^1 \times \mathbb{S}^{n - 2} \ni (x_1, x_2, x, \mu, y) \mapsto \big(x_1, x_2, x, \mu + i (y_1 \mu_\perp + (y_2, \dotsc, y_{n - 1}))\big) \in L_1.
	\]
	where we used the identification  $T M_0 \cong M_0 \times \mathbb{R}^{n - 2}$ and wrote $\mu_\perp$ for the vector obtained by applying the oriented rotation by $\frac{\pi}{2}$ to $\mu$ in $\mathbb{R}^2$.
 
 This argument also shows that for general $(M_0, g_0)$, $L_1$ has the required fibre.
\end{proof}

There is a natural differential operator $\mathbb{X}$ on $L$ induced from the geodesic vector field $X$ of $\mathbb{R}^2 \times M_0$, that we sometimes refer to as the \emph{complexified geodesic vector field}. For a function $f \in C^\infty(L)$, set
\begin{align}\label{eq:complex-X}
\begin{split}
	\mathbb{X}(x_1, x_2, x, \mu + i \nu) &f(x_1, x_2, x, \mu + i\nu) \\
	&:= \frac{d}{dt}\Big|_{t = 0} f\big(\gamma_\mu(t), P_{\gamma_{\mu}}(t) (\mu + i\nu)\big) + i\frac{d}{dt}\Big|_{t = 0} f\big(\gamma_\nu(t), P_{\gamma_{\nu}}(t) (\mu + i\nu)\big),
\end{split}
\end{align}
where $\gamma_{\mu}$ and $\gamma_{\nu}$ are the geodesics starting at $(x_1, x_2, x)$ with speeds $\mu$ and $\nu$, respectively, and $P_{\gamma}$ denotes Levi-Civita parallel transport with respect to the curve $\gamma$. Note that this parallel transport leaves $L$ invariant since parallel transport preserves norms and angles.
\medskip

In what follows we assume $(M_0, g_0)$ is simple and turn to an auxiliary result that solves a certain matrix $\overline{\partial}$-equation leafwise, for a local family of bicharacteristic leaves parameterised from the boundary. Enlarge $(M_0, g_0)$ to another simple manifold $(N_0, g_{0})$ which is in turn embedded into a further manifold as in \S \ref{ssec:extensions}. This parameterisation is done by fixing a small neighbourhood $U$ of $\mu_0 \in \mathbb{S}^1$, a small neighbourhood of a point $p_0$ in $\partial(D \times N_0)$, and inwards pointing unit vectors $\nu$ at $(x_1, x_2, x) = p \in V$ of the slice $\{(x_1, x_2) + \mu^\perp\} \times N_0$ (see Figure \ref{fig:transversal}). The $\overline{\partial}$-equation is then solved for the bicharacteristic leaf generated by $(\mu, \nu)$ at $p$, which can be naturally identified with $\mathbb{C}$ using geodesics defined by $\mu$ and $\nu$. In what follows $D \subset \mathbb{R}^2$ denotes a closed disk centred at zero (more general domains can be taken, but for simplicity we restrict to the disk setting).

\begin{lemma}\label{lemma:parametric-del-bar}
	Let $\mu_0 \in \mathbb{S}^1 \subset \mathbb{R}^2$ and $p_0 := (x_{10}, x_{20}, x_0) \in \mc{B} := (\partial D \times N_0^\circ) \sqcup (D^\circ \times \partial N_0)$. Let $A$ be a connection with compact support in the interior of $D \times N_0$. Let $U \subset \mathbb{S}^1$ be an open set diffeomorphic to an interval, containing $\mu_0$, and let $V$ be an open neighbourhood of $p_0$ in $\mc{B}$ diffeomorphic to a disk. For $\mu \in U$, and $p \in V$, denote by $\mu^\perp \subset \mathbb{R}^2$ the orthogonal of $\mu$ and consider the set of inward pointing or tangent unit vectors $\mc{D}_{\leq 0}(\mu, p)$ of $\{(x_1, x_2) + \mu^\perp\} \times N_0$ at $p = (x_1, x_2, x)$, i.e.
    \[
        \mc{D}_{\leq 0}(\mu, p) = \{v \in T_p(\{(x_1, x_2) + \mu^\perp\} \times N) \mid |v|_{e \oplus g} = 1,\quad v\,\,\mathrm{inward\,\, pointing\,\,or\,\,tangent\,\,to\,\,} \mc{B}\}.
    \]
    This forms a trivial disk fibre bundle $\mc{D}_{\leq 0}$ over $U \times V$ (see Figure \ref{fig:transversal}). Then, there exists 
	\[
            C \in C^\infty\big(\mc{D}_{\leq 0} \times \mathbb{C}; \GL_r(\mathbb{C})\big)
        \] 
	satisfying for $(\mu, p, \nu) \in \mc{D}_{\leq 0}$
	\begin{equation}\label{eq:del-bar}
		\partial_{\bar{z}} C(\mu, p, \nu, z) + A\big((x_1, x_2) + s\mu + t\nu_\perp, \exp_p\big(t(\nu - \nu_\perp)\big)\big)(\partial_{\bar{z}}) \cdot C(\mu, p, \nu, z) = 0.
	\end{equation}
	Here $z = s + i t \in \mathbb{C}$, $\partial_{\bar{z}} = \frac{1}{2}(\partial_s + i \partial_t)$, $\nu_\perp$ is the component of $\nu$ in the $\mu^\perp$ direction, and $\exp_p$ is the exponential map of $(N_0, g_0)$.
\end{lemma}
	Here we use the convention that $\exp_p\big(t(\nu - \nu_\perp)\big)$ is defined outside of the domain of $\exp_p$ using the extended geodesic flow. We note that in the somewhat standard notation within the area of geometric inverse problems, we could write $ \mc{D}_{\leq 0}(\mu, p) = \partial_- S_p(\{(x_1, x_2) + \mu^\perp\} \times N_0)$, also noting that we include the boundary tangent vectors.

 \begin{center}
\begin{figure}[htbp!]
\includegraphics[scale=1]{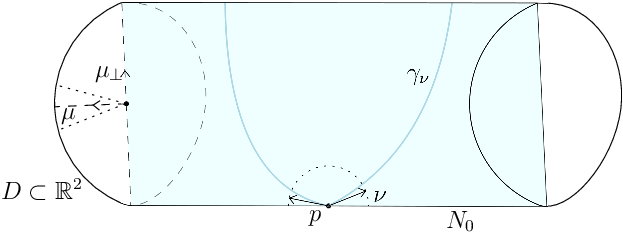}
\caption{\label{fig:transversal} Visualisation of $\mc{D}_{\leq 0}(\mu, p)$ when $(M_0, g_0)$ is the interval with Euclidean metric; in shaded blue is the manifold orthogonal to $\mu$, and in dotted lines is the set of possible vectors in the neighbourhood $U \times V$; $\gamma_\nu$ is the geodesic generated by $\nu$.}
\end{figure}
\end{center}

\begin{proof}
	We will use the results of Nakamura-Uhlmann \cite[Section 5]{nakamura-uhlmann-02}.
	\medskip
	
	\emph{Step 1: local solution.} Near any $(\mu, p, \nu) \in \mc{D}_{\leq 0}$ a smooth invertible solution $C$ exists by \cite[Lemma 5.8]{nakamura-uhlmann-02}. In the mentioned reference the authors deal with a one-dimensional space of parameters, but the proof works more generally for any finite number of dimensions. 
	\medskip
	
	\emph{Step 2: gluing local solutions.} We follow a general construction similar to \cite[page 335-336]{nakamura-uhlmann-02}. By slightly enlarging the parameter space to $\mc{D}$ to contain all nearly tangential directions and slightly enlarging the disks $U$ and $V$, by compactness we may assume there is a finite cover by open sets $\mc{U}_{i = 1}^k$ of $\mc{D}_{\leq 0}$ and smooth invertible solutions $C_i$ of \eqref{eq:del-bar} over $\mc{U}_i$. For any $i$ and $j$, $C_{ij} := C_i^{-1} C_j$ satisfies on $\mc{U}_i \cap \mc{U}_j$
	\[
		\partial_{\bar{z}} (C_i^{-1} C_j) = -C_i^{-1} (\partial_{\bar{z}} C_i) C_i^{-1} C_j + C_i^{-1} \partial_{\bar{z}} C_j = 0,
	\] 
	where in the last equality we used \eqref{eq:del-bar} (and we dropped the parameters for simplicity). Define
	\[
		\mc{G} := \{F \in C^\infty(\mathbb{C}, \GL_r(\mathbb{C})) \mid \partial_{\bar{z}} F = 0\},
	\]
	equipped with the usual topology of $C^\infty(\mathbb{C}, \GL_r(\mathbb{C}))$ of convergence on compact sets. Using the cover by $(\mc{U}_i)_{i = 1}^k$ and the $\mc{G}$-valued transition maps $C_{ij}$, we define the $\mc{G}$-fibre bundle $\mc{F}$ over $\mc{D}$. Since $\mc{D}_{\leq 0}$ is disk bundle over a contractible space, it is contractible itself, and we may assume that so is $\mc{D}$; by \cite[Corollary 11.6]{steenrod-99} we hence know that $\mc{F}$ has a smooth section $H$. If $H_i$ denotes $H$ in the local trivialisation over $\mc{U}_i$, then on  $\mc{U}_i \cap \mc{U}_j$ by definition $H_i = C_{ij} H_j$. Define on each $\mc{U}_i$, $C := C_i H_i$, so we see that 
	\[
		C_i H_i = C_i C_{ij} H_j = C_j H_j, \quad \mathrm{on} \quad \mc{U}_i \cap \mc{U}_j.
	\]
	Therefore, $C|_{\mc{D}_{\leq 0}}$ is a well-defined smooth global solution to \eqref{eq:del-bar}, which completes the proof.
\end{proof}

Next, we turn to the main result of this section. For brevity, we will sometimes denote the points of $L$ as $\mu + i\nu$ (the coordinate $(x_1, x_2, x)$ will be clear from context). To summarise, this result reduces Theorem \ref{thm:main-theorem-simple} to a (complex) transport problem on $L$. In the proof, we use the family of solutions to the leafwise $\overline{\partial}$-equation constructed in Lemma \ref{lemma:parametric-del-bar} to construct CGOs (Step 1), and then use the information given by the assumption (equality of DN maps) and the integral (Alessandrini) identity to modify these leafwise solutions so that they by uniqueness smoothly glue together (Steps 2 and 3) and indeed satisfy a complex transport equation on $L$.

\begin{theorem}\label{thm:reduction-to-complex-x-ray-connection}
    	Let $A_1, A_2$ and $Q_1, Q_2$ be as in Theorem \ref{thm:main-theorem-simple} and assume they are in gauges as in Proposition \ref{proposition:boundary-determination}. There exists a smooth invertible section $G \in C^\infty(L_1|_{D \times M_0}, \GL_r(\mathbb{C}))$ satisfying for all  $(x_1, x_2, x, \mu + i \nu) \in L_1|_{D \times M_0}$
    	\begin{equation}\label{eq:homomorphism-connection-complex-transport}
    		\mathbb{X}(\mu + i\nu) G(\mu + i\nu) + A_1(x_1, x_2, x)(\mu + i\nu) G(\mu + i\nu) - G(\mu + i\nu) A_2(x_1, x_2, x)(\mu + i\nu) = 0.
    	\end{equation}
    	Moreover, we have $G \equiv \id$ for base points in the connected component of $\partial (D \times M_0)$ inside $D \times M_0 \setminus M$.
	
	If we extend $G$ by $0$-homogeneity to $L$, then it satisfies \eqref{eq:homomorphism-connection-complex-transport} for $(x_1, x_2, x, \mu + i \nu) \in L|_{D \times M_0}$.
\end{theorem}
\begin{proof}
	We divide the proof into steps.
	\medskip
	
	\emph{Step 1: application of CGOs.} Fix $\mu \in \mathbb{S}^1$, $p \in \mc{B}$, and $\nu \in \mc{D}_{\leq 0}(\mu, p)$, and open sets $\mu \in U$ and $p \in V$, as in the notation of Lemma \ref{lemma:parametric-del-bar} (see Figure \ref{fig:transversal}). Without loss of generality (by applying a rotation in $\mathbb{R}^2$), we may assume that $\mu = \e_1 = (1, 0) \in \mathbb{R}^2$. Consider the transversal manifolds $N_\mu = [-T, T] \times N_0$ and $M_\mu = [-T, T] \times M_0$ for some sufficiently large $T$ (depending only on $M$). By Lemma \ref{lemma:parametric-del-bar}, there is a solution $C_1 \in C^\infty(\mc{D}_{\leq 0} \times \mathbb{C}; \GL_r(\mathbb{C}))$ to the equation \eqref{eq:del-bar}. Observe that this is the same equation as in \eqref{eq:transport-1}, that is
    \begin{equation}\label{eq:transport-1-simple}
        \partial_{\bar{z}} C_1 + A_1(\partial_{\bar{z}}) C_1 = 0.
    \end{equation}
    Thus, by the construction laid out in \S \ref{ssec:CGO}, there are smooth solutions $U_1$ to $\mc{L}_{A_1, Q_1} U_1 = 0$ on $[-T, T] \times M_\mu$ of the form
	\[
		U_1(x_1, x_2, x) = e^{-\frac{x_\mu + i r}{h}} (|g|^{-\frac{1}{4}} C_1(x_\mu, r, \theta) + R_1),
	\]
	where $x_\mu = (x_1, x_2) \cdot \mu = x_1$, $h > 0$ is a small (semiclassical) parameter, $(r, \theta)$ where $\theta \in \mc{D}_{\leq 0}(\mu, p)$ are seen as polar coordinates in $M_\mu$. Similarly, we obtain smooth solutions $U_2$ to $\mc{L}_{-A_2^*, Q_2^*} U_2 = 0$, of the form
	\[
		U_2(x_1, x_2, x) = e^{-\frac{-x_\mu + i r}{h}} (|g|^{-\frac{1}{4}} C_2(x_\mu, r, \theta) + R_2),
	\]
	where $C_2 \in C^\infty(\mc{D}_{\leq 0} \times \mathbb{C}; \GL_r(\mathbb{C}))$ satisfies \eqref{eq:transport-2} with $-A_2^*$ instead of $A$
        \begin{equation}\label{eq:transport-2-simple}
            \partial_z C_2 - A_2^*(\partial_z) C_2 = 0.
        \end{equation}
    The remainders satisfy 
    \begin{equation}\label{eq:remainder}
		\|R_i\|_{L^2} = \mc{O}(h), \quad \|R_i\|_{H^1} = \mc{O}(1), \quad h \to 0,\quad i = 1, 2,
    \end{equation}
    Write $U_i|_{\partial M} =: F_i$ for $i = 1, 2$, and use Proposition \ref{prop:integral-identity}:
	\begin{align}\label{eq:integral-identity-inside}
	\begin{split}
		0 &= \langle{(\Lambda_{A_1, Q_1} - \Lambda_{A_2, Q_2})F_1, F_2}\rangle_{L^2}\\ 
		&= \underbrace{\langle{(Q_1 - Q_2 + |A_2|^2 - |A_1|^2) U_1, U_2}\rangle_{L^2}}_{T_1} + \underbrace{\langle{U_1\cdot dU_2^* - dU_1\cdot U_2^*, A_1^* - A_2^*}\rangle_{L^2}}_{T_2}.
	\end{split}
	\end{align}
	For the first term on the right hand side we get
	\begin{align}\label{eq:T-1}
	\begin{split}
		T_1 &= \int_M \Tr\left((Q_1 - Q_2 + |A_2|^2 - |A_1|^2)U_1 U_2^*\right)\,d\vol_g\\ 
		&= \int_M \Tr\left((Q_1 - Q_2 + |A_2|^2 - |A_1|^2)(|g|^{-\frac{1}{4}}C_1 b + R_1) (|g|^{-\frac{1}{4}} C_2^* + R_2^*)\right)\, d\vol_g = \mc{O}(1)
	\end{split}
	\end{align}
	using \eqref{eq:remainder}. For the other term, we similarly compute
	\begin{align}\label{eq:T-2}
	\begin{split}
		T_2 &= \int_M \Tr \left\langle{U_1\cdot dU_2^* - dU_1\cdot U_2^*, A_1^* - A_2^*}\right\rangle_g \,d\vol_g\\
		&= \mc{O}(1) + 2h^{-1} \int_M \Tr\left\langle(|g|^{-\frac{1}{4}}C_1 + R_1)(dx_\mu + i dr)(|g|^{-\frac{1}{4}} C_2^* + R_2^*), A_1^* - A_2^*\right\rangle_g\, d\vol_g\\
		&= \mc{O}(1) - 2h^{-1}\int_M \Tr\left(C_1 C_2^* (\widetilde{A}_1 + i \widetilde{A}_r)\right)\, dx_\mu dr d \theta,
	\end{split}
	\end{align}
	where we denote $\widetilde{A} := A_2 - A_1$, in the second line we used that the terms not containing derivatives of $e^{\frac{x_\mu + ir}{h}}$ are bounded in $L^2$, in the third line we used \eqref{eq:remainder} and that $d\vol_g = |g|^{\frac{1}{2}}\, dx_\mu dr d\theta$. Therefore
	\[
		0 = \lim_{h \to 0} h(T_1 + T_2) = - 2\int_M \Tr\left(C_1 C_2^* (\widetilde{A}_1 + i \widetilde{A}_r)\right)\, dx_\mu dr d \theta.
	\]
	
	By replacing $C_1$ with $C_1 b$, where $b$ is only a function of the polar coordinate $\theta$ (this is allowed as $C_1 b$ also satisfies \eqref{eq:transport-1}), and further taking a sequence of smooth functions $(b_i(\theta))_{i = 1}^\infty$ such that $b_i(\theta) d\theta \to \delta(\nu)$ as a distribution, where $\delta(\nu)$ denotes the Dirac delta distribution at $\nu$ (note here that the integral can also be interpreted over $D \times N_0$ as $\supp(\widetilde{A}) \subset M$), we get
	\begin{equation}\label{eq:CGO-limit-inside-connections}	
		0 = \int_{B(\nu)} \Tr\left(C_1 C_2^* (\widetilde{A}_1 + i \widetilde{A}_r)\right) dx_\mu dr,
	\end{equation}
	where $B(\nu)$ denotes the bicharacteristic leaf determined by $(\mu, p, \nu)$; more explicitly 
	\begin{equation}\label{eq:leaf-explicit}
		B(\nu) = \left\{\big((y_1, y_2) + s\mu + t\nu_\perp, \exp_p(t(\nu - \nu_\perp)\big) \mid |s|, |t| \leq T\right\},
	\end{equation}
	where $p = (y_1, y_2, y)$, $\nu_\perp$ is the component of $\nu$ in direction of $\mu^\perp$, and $T > 0$ is large enough depending on the diameter of $D \times N_0$.
 
 Since we may replace $C_1$ by $C_1 c$, where $c$ denotes an arbitrary constant coefficient matrix, we get
	\[	
		0 = \int_{B(\nu)} C_1 C_2^* (\widetilde{A}_1 + i \widetilde{A}_r)\, dx_\mu dr.
	\]
	Finally, since we may replace $C_1$ by $C_1 H$ where $H = H(x_\mu, r)$ is an arbitrary holomorphic matrix that does not depend on $\theta$ (as $C_1 H$ also satisfies \eqref{eq:transport-1-simple}), we get
	\begin{equation}\label{eq:CGO-end-identity}
		0 = \int_{B(\nu)} H C_2^* (\widetilde{A}_1 + i \widetilde{A}_r)C_1 \, dx_\mu dr = -i\int_{\partial B(\nu)} HC_2^*C_1\, dz,
	\end{equation}
	where in the second equality we denoted $z = x_\mu + ir$, we used Stokes' theorem and 
	\begin{equation}\label{eq:C_2^*C_1}
		\partial_{\bar{z}} (C_2^* C_1) = \frac{1}{2} C_2^*(\widetilde{A}_1 + i \widetilde{A}_r) C_1,
	\end{equation}
 which follows from \eqref{eq:transport-1-simple} and \eqref{eq:transport-2-simple}.
	\medskip
	
	\emph{Step 2: gluing the solutions.} By \cite[Lemmas 4.6 and 4.7]{cekic-17} we deduce from \eqref{eq:CGO-end-identity} that there is a holomorphic, invertible matrix function $F$ on the interior of $B(\nu)$, continuous up to the boundary, such that $F^{-1}|_{\partial B(\nu)} = C_2^*C_1|_{\partial B(\nu)}$. In fact, by the Plemelj-Sokhotski formula we have
	\begin{equation}\label{eq:plemelj-sokhotski}
		F^{-1}(\mu, p, \nu, z) = \frac{1}{2\pi i} \int_{\partial B(\mu, p, \nu)} \frac{C_2^*(\mu, p, \nu, \zeta) C_1(\mu, p, \nu, \zeta)}{\zeta - z}\, d\zeta, \quad z \in B(\nu)^\circ.
	\end{equation}
	Here and in what follows, we make an identification of the bicharacteristic leaf defined by $(\mu, p, \nu)$ with a subset of $\mathbb{C}$ using the coordinates in \eqref{eq:leaf-explicit}. From this formula, and since $\partial B(\nu)$, $C_1$, and $C_2$ vary smoothly in $\nu$, we deduce that $F^{-1}(\mu, p, \nu, z)$ also depends smoothly on $(\mu, p, \nu) \in \mc{D}_{\leq 0}$ and $z$ in some large disk in $\mathbb{C}$.
	 
Setting $G := C_1 F C_2^*$ we see that
	\begin{equation}\label{eq:G-equation}
		\partial_{\bar{z}} G + A_1 (\partial_{\bar{z}})G - GA_2(\partial_{\bar{z}}) = 0, \quad G|_{\partial B(\nu)} = \id, \quad z \in B(\nu)^\circ.
	\end{equation}
	Since $C_2^*C_1$ is in fact holomorphic in $\mathbb{C} \setminus B(\nu)$ by \eqref{eq:CGO-end-identity} (and so in particular near $\partial B(\nu)$), we may shift the contour in \eqref{eq:plemelj-sokhotski} and so by the preceding equation (and since $A_i$ are compactly supported inside the interior of $D \times N$) get that $G$ extends holomorphically past $\partial B(\nu)$ to $\mathbb{C}$. Since $G = \id$ on $\partial B(\nu)$ we conclude that $G \equiv \id$ in the exterior of $B(\nu)$; moreover, by uniqueness of solutions to \eqref{eq:G-equation} we have $G \equiv \id$ on the connected component of $\partial B(\nu)$ for base points in the complement of $M$ (as $\widetilde{A} \equiv 0$ there).
	
	By possibly shrinking $U$ and $V$ introduced in Step 1, we have thus obtained a pointwise invertible smooth solution $G_{U, V} := G \in C^\infty(\mc{D}_{\leq 0} \times \mathbb{C}, \GL_r(\mathbb{C}))$ satisfying \eqref{eq:G-equation} with $G \equiv \id$ in the connected component of $\partial B(\nu)$ of base points in the complement of $M$. Since $(\mu, p) \in \mathbb{S}^1 \times \mc{B}$ were arbitrary, there is a covering $(U_i\times V_i)_{i \in I}$ of $\mathbb{S}^1 \times \mc{B}$ and corresponding solutions $G_i$ of \eqref{eq:G-equation} over $\mc{D}_{\leq 0}|_{U_i \times V_i} \times \mathbb{C}$. Now observe that $G_i$ and $G_j$ agree on the intersection of their domains: indeed, since $G_j \equiv G_i \equiv \id$ for large $|z|$, by uniqueness of solutions to \eqref{eq:G-equation}, $G_i \equiv G_j$. Here uniqueness may be seen by the \emph{unique continuation principle}: applying $\partial_z$ and using $-4\partial_z \partial_{\bar{z}} = - \Delta_{\mathbb{C}}$, we are reduced to uniqueness of solutions to principally diagonal elliptic systems. By gluing together $(G_i)_{i \in I}$, we obtain $G \in C^\infty(\mc{D}_{\leq 0} \times \mathbb{C}; \GL_r(\mathbb{C}))$ satisfying \eqref{eq:G-equation}. 
	
	In fact, by repeating the same procedure as above with outwards instead of inwards pointing vectors, and again using the uniqueness argument from the previous paragraph, we obtain a smooth invertible $G \in C^\infty(\mc{D} \times \mathbb{C}; \GL_r(\mathbb{C}))$ satisfying \eqref{eq:G-equation}, where $\mc{D}(\mu, p)$ is the set of all unit vectors tangent to $M_\mu$  (technically, to obtain a smooth solution we would need to consider outwards pointing vectors together with nearly tangential ones; the fibre would still be contractible in this case which is needed for the appropriate version of Lemma \ref{lemma:parametric-del-bar}).
	 
	\medskip
	
	\emph{Step 3: extension to the interior.} Given $(x_1, x_2, x, \mu + i\nu) \in L_1|_{D \times M_0}$, let $p = ((x_1, x_2) - t \mu, x) \in \partial D \times N_0^\circ$ be the first point of intersection of the geodesic $\eta$ starting at $(x_1, x_2, x)$ in direction $-\mu$ with the boundary of $D \times N_0$. We define
	\[
		\widetilde{G}(x_1, x_2, x, \mu + i \nu) := G(\mu, p, \nu, t),
	\]
	where since $\eta$ leaves $x$ fixed we may identify $\nu$ with its parallel transport along $\eta$. Since the geodesic intersects the boundary transversely, in this way we obtain a smooth invertible map on $L_1|_{D \times M_0}$. We claim that it satisfies \eqref{eq:homomorphism-connection-complex-transport}. Indeed, writing $\mu_\perp$ for the (oriented) unit orthogonal to $\mu$, and $\nu = a\mu_\perp + v$, where $v \in T_{x}M_0$ and $a \in \mathbb{R}$, by definition we have
	\begin{align*}
		&\mathbb{X}(\mu + i\nu) \widetilde{G}(x_1, x_2, x, \mu + i\nu)\\ 
		&= \partial_s|_{s = 0} \widetilde{G}\big((x_1, x_2) + s\mu, x, \mu + i\nu\big) + i\partial_s|_{s = 0} \widetilde{G}\big((x_1, x_2) + a\mu_\perp s, \gamma(s), \mu + i(a\mu_\perp + \dot{\gamma}(s))\big)\\
		&= \partial_s|_{s = 0} G\big(\mu, p, \nu, t + s\big) + i \partial_s|_{s = 0} G\big(\mu,\underbrace{(x_1, x_2) + a\mu_\perp s - t(s) \mu, \gamma(s)}_{p(s) :=}, a \mu_\perp + \dot{\gamma}(s), t(s)\big),
	\end{align*}
	where $\gamma(s)$ denotes the geodesic starting at $x$ with speed $v$, and $t(s)$ is some time depending smoothly on $s$ such that $p(s) \in \partial D \times N_0^\circ$. Now, the bicharacteristic leaf at $p$ in direction $\mu + i\nu$ coincides with the one at $p(s)$ in direction $\mu + i(a\mu_\perp + \dot{\gamma}(s))$, so by uniqueness of solutions to \eqref{eq:G-equation}, we conclude that $G(\mu, p, \nu, z)$ agrees with $G(\mu, p(s), a\mu_\perp + \dot{\gamma}(s), z - z_0(s))$ for all $z \in \mathbb{C}$ for some fixed $z_0(s)$ depending smoothly on $s$. In fact, we have $z_0(s) = t + is - t(s)$, and so
	\begin{align*}
		\mathbb{X}(\mu + i\nu) &\widetilde{G}(x_1, x_2, x, \mu + i\nu) = \partial_s|_{s = 0} G\big(\mu, p, \nu, t + s\big) + i\partial_s|_{s = 0} G\big(\mu, p, \nu, t + is\big)\\
		&= -A_1(x_1, x_2, x) (\mu + i\nu) \widetilde{G}(x_1, x_2, x) + \widetilde{G}(x_1, x_2, x) A_2(x_1, x_2, x) (\mu + i\nu),
	\end{align*}
	where in the second line we used \eqref{eq:G-equation}. This proves the claim. 
	
	We now re-label $\widetilde{G}$ by $G$; note that by construction $G = \id$ on the connected component of $\partial (\mathbb{R}^2 \times M_0)$ in the complement of $M$.
\medskip

\emph{Step 4: extension by homogeneity.} Finally, we show the last statement of the theorem. It suffices to consider $r > 0$ and $\mu + i\nu \in L_1$, and to compute
\[
	\mathbb{X}(r(\mu + i\nu))G(r(\mu + i\nu)) = r \mathbb{X}(\mu + i\nu) G(\mu + i\nu),
\] 
where in the equality we used $0$-homogeneity of $G$, the definition \eqref{eq:complex-X} of $\mathbb{X}$, as well as that $\gamma_{r\mu}(t) = \gamma_\mu(rt)$ and $\gamma_{r\nu}(t) = \gamma_{\nu}(rt)$ (with the notation as in \eqref{eq:complex-X}). Then it is left to use \eqref{eq:homomorphism-connection-complex-transport} and the linearity of $A_1$ and $A_2$. This completes the proof.
\end{proof}

There are certain symmetries that the solutions $G$ from Theorem \ref{thm:reduction-to-complex-x-ray-connection} have to satisfy. In the following, we use the notation $G^{-*} := (G^{-1})^*$ for the Hermitian conjugate of the inverse of $G$.

\begin{lemma}\label{lemma:symmetries}
	For any $(x_1, x_2, x, \mu + i\nu) \in L$, we have
 \begin{enumerate}[itemsep=5pt]
	\item[1.] We have the relation
        \[
            G(x_1, x_2, x, \mu + i\nu) = G(x_1, x_2, x, -\mu - i\nu).
        \]
	\item[2.] If $\mathrm{span}(\mu, \nu) = \mathbb{R}^2$ and $\theta \in \mathbb{R}$, then
	\[
		G(x_1, x_2, x, \mu + i\nu) = G(x_1, x_2, x, e^{i\theta}(\mu + i\nu)).
	\]
        \item[3.] If in addition we know that $A_1$ and $A_2$ are unitary, then
        \[
            G(x_1, x_2, x, \mu + i\nu) = G^{-*}(x_1, x_2, x, \mu - i\nu).
        \]
\end{enumerate}
\end{lemma}
\begin{proof}
	It suffices to consider $\mu + i\nu \in L_1$. For the first identity, we observe that both $\mu + i\nu$ and $-\mu - i\nu$ generate the same bicharacteristic leaf (see \eqref{eq:leaf-explicit}); then changing the $2$-frame corresponds to a linear change of coordinates in $\mathbb{C}$, so the claim follows by uniqueness of solutions to \eqref{eq:G-equation} (and the fact that $A_1$ and $A_2$ act linearly).
	
	The second equality follows again similarly, using that $e^{i\theta}(\mu + i\nu)$ generates the same bicharacteristic leaf as $\mu + i\nu$, and a linear change of coordinates $z \mapsto e^{i\theta} z$ in that leaf. 

    To see the final identity, observe that since $G(\mu - i \nu)$ satisfies along the bicharacteristic leaf \eqref{eq:homomorphism-connection-complex-transport} with $\partial_z$ instead of $\partial_{\bar{z}}$, we simply compute that $G^{-*}$ (using that $A_i^* = -A_i$ for $i = 1, 2$) satisfies the same equation as $G$; so again by uniqueness of solutions, the identity follows.
\end{proof}

We now give the analogue of Theorem \ref{thm:reduction-to-complex-x-ray-connection} for the case of potentials.

\begin{theorem}\label{thm:reduction-to-complex-x-ray-potential}
	Assume we are in the setting of Theorem \ref{thm:main-theorem-simple} with the additional assumption that $A_1 = A_2 =: A$. Then, there exists $F \in C^\infty(L, \mathbb{C}^{r \times r})$ satisfying
 	\begin{equation}\label{eq:endomorphism-potential-complex-transport}
 		\mathbb{X}(\mu + i\nu) F(\mu + i\nu) + [A(\mu + i\nu), F(\mu + i\nu)] = Q_1 - Q_2,
 	\end{equation}
 	and moreover $F \equiv 0$ for base points in the connected component of $\partial (\mathbb{R}^2 \times M_0)$ inside $\mathbb{R}^2 \times M_0 \setminus M$.
\end{theorem}
\begin{proof}
	Since the proof is similar to the proof of Theorem \ref{thm:reduction-to-complex-x-ray-connection}, we will use the same notation and only highlight the differences.
	
	\medskip
	\emph{Step 1: application of CGOs.} In this situation we may take $C_1 =: C$ as before and notice that the transport equation \eqref{eq:transport-2-simple} is then satisfied by $C_2 = C^{-*}$. In fact, let $H = H(x_\mu, r)$ be an arbitrary holomorphic matrix and use $C H$ for the solution of the transport equation \eqref{eq:transport-1-simple} (but keep $C_2 = C^{-*}$). Then the construction of CGOs is the same; we plug those inside Proposition \ref{prop:integral-identity}, which now simplifies and gives
	\begin{align}\label{eq:integral-identity-inside-potentials}
	\begin{split}
		0 &= \langle{(\Lambda_{A, Q_1} - \Lambda_{A, Q_2}) F_1, F_2}\rangle_{L^2} = \langle{(Q_1 - Q_2) U_1, U_2}\rangle_{L^2}\\
		&= \int_M \Tr\left((Q_1 - Q_2)(|g|^{-\frac{1}{4}}CH + R_1) (|g|^{-\frac{1}{4}} C^{-1} + R_2^*)\right)\, d\vol_g\\
		&= \mc{O}(h) + \int_M \Tr\left((Q_1 - Q_2)CHC^{-1}\right) \,dx_\mu dr d\theta,
	\end{split}
	\end{align}
	where in the last line we used \eqref{eq:remainder}, and wrote $d\vol_g = |g|^{\frac{1}{2}} dx_\mu dr d\theta$. Letting $h \to 0$ the first term in the last line of \eqref{eq:integral-identity-inside-potentials} vanishes. Then, taking $Hc$ instead of $H$ for an arbitrary constant matrix $c$ we get
	\[
		\int_M C^{-1}(Q_1 - Q_2)CH \,dx_\mu dr d\theta = 0.
	\]
	Changing $H$ to $Hb_i$ where $(b_i(\theta))_{i = 1}^\infty$ is a sequence of smooth functions such that $b_i(\theta) d\theta \to \delta(\nu)$, we get 
		\begin{equation}\label{eq:potentials-identity-plane}
			\int_{B(\nu)} C^{-1}(Q_1 - Q_2)CH \,dx_\mu dr = 0.
		\end{equation}
		Let us now consider the solution $E \in C^\infty(\mathbb{C}, \mathbb{C}^{r \times r})$ to
		\begin{equation}\label{eq:E-equation}
			\partial_{\bar{z}} E = C^{-1} (Q_1 - Q_2)C := f, \quad z \in B(\nu).
		\end{equation}
		This can be solved uniquely by applying the \emph{Cauchy operator}, i.e.
		\[
			E(z) := \Pi (f)(z) := \frac{1}{2\pi i} \int_{\mathbb{C}} \frac{f(\tau)}{\tau - z}\, d\tau d\bar{\tau},
		\]
		  or alternatively by applying Lemma \ref{lemma:del-bar-source} directly. In fact, since $f = C^{-1} (Q_1 - Q_2) C$ depends locally smoothly on $(\mu, p, \nu)$, so does $E(z)$, and we obtain $E \in C^\infty(\mc{D}_{\leq 0} \times \mathbb{C}, \mathbb{C}^{r \times r})$.
		
		Therefore, integrating by parts in \eqref{eq:potentials-identity-plane} we get for every holomorphic matrix $H$
		\begin{equation}\label{eq:potentials-identity-plane-final}
			\int_{B(\nu)} E H \,dz = 0.
		\end{equation}
		
		\medskip
		
		\emph{Step 2: gluing the solutions.} We will first modify and then glue together $E \in C^\infty(\mc{D}_{\leq 0} \times \mathbb{C}, \mathbb{C}^{r \times r})$ constructed in the previous step. By \eqref{eq:potentials-identity-plane-final}, and the Plemelj-Sokhotski formula (as in \eqref{eq:plemelj-sokhotski}), there is a holomorphic matrix $E_0$ that agrees with $E$ restricted to the boundary $\partial B(\nu)$, given by the formula
	\begin{equation*}
		E_0(\mu, p, \nu, z) = \frac{1}{2\pi i} \int_{\partial B(\mu, p, \nu)} \frac{E(\mu, p, \nu, \zeta)}{\zeta - z}\, d\zeta, \quad z \in B(\nu)^\circ.
	\end{equation*}
	As before (see the paragraph after \eqref{eq:plemelj-sokhotski}), $E_0$ depends locally smoothly on $(\mu, p, \nu)$. Moreover, we may shift the contour in the above integral by using the holomorphicity of $E$ for large $|z|$ (follows from \eqref{eq:E-equation} and the fact that $Q_1 = Q_2$ outside $M$, so also outside $B(\nu)$), and so we get that
	\[
		E_0 \in C^\infty(\mc{D}_{\leq 0} \times \mathbb{C}, \mathbb{C}^{r \times r}), \quad (E - E_0)|_{\mathbb{C} \setminus B(\nu)} = 0.
	\]
	We set $F := C(E - E_0)C^{-1}$, which then satisfies
	\[
		\partial_{\bar{z}} F = \partial_{\bar{z}}C \cdot (E - E_0)C^{-1} + C \partial_{\bar{z}}E \cdot C^{-1} - C (E - E_0)C^{-1} \partial_{\bar{z}}C \cdot C^{-1} = [F, A(\partial_{\bar{z}})] + Q_1 - Q_2. 
	\]	
	where in the first equality we used $\partial_{\bar{z}} E_0 = 0$, and in the second equality we used \eqref{eq:transport-1-simple} as well as \eqref{eq:E-equation}. Equivalently,
	\begin{equation}\label{eq:F-equation}
		(\partial_{\bar{z}} + [A(\partial_{\bar{z}}), \bullet]) F = Q_1 - Q_2, \quad F|_{\mathbb{C} \setminus B(\nu)} = 0.
	\end{equation}
	Observe that by the unique continuation principle, $F$ is the unique solution to \eqref{eq:F-equation}, and so similarly to Step 2 of Theorem \ref{thm:reduction-to-complex-x-ray-connection} $F$ glues well on overlaps to give $F \in C^\infty(\mc{D} \times \mathbb{C}, \mathbb{C}^{r \times r})$ satisfying \eqref{eq:F-equation}. Extension to the interior works entirely the same as before to give \eqref{eq:endomorphism-potential-complex-transport} for $\mu + i\nu \in L_1$.
	\medskip
	
	\emph{Step 3: extension by homogeneity.} To keep the equation \eqref{eq:F-equation} scale invariant, we extend $F$ by $-1$ homogeneity to $L$:
	\[
		F(x_1, x_2, x, r(\mu + i\nu)) := \frac{1}{r} F(x_1, x_2, x, \mu + i\nu), \quad r > 0, (x_1, x_2, x, \mu + i\nu) \in L_1.
	\]
	Then by Step 2 above, and the same computation as in Step 4 of Theorem \ref{thm:reduction-to-complex-x-ray-connection}, we obtain \eqref{eq:endomorphism-potential-complex-transport}, which completes the proof.
\end{proof}

\section{Uniqueness in the setting of simple transversal geometry}\label{sec:simple-transversal-geometry}

In this section we prove Theorem \ref{thm:main-theorem-simple} based on the uniqueness for the complex transport problem obtained in Theorem \ref{thm:reduction-to-complex-x-ray-connection}. In fact, we will show

\begin{lemma}\label{lemma:connection-uniqueness}
	Assume $G$ satisfies the equation stated in Theorem \ref{thm:reduction-to-complex-x-ray-connection}. Then $G$ is independent of the $\mu + i\nu$ variable. More precisely, there exists $G_0 \in C^\infty(D \times M_0, \GL_r(\mathbb{C}))$, such that 
	\[
		G(x_1, x_2, x, \mu + i\nu) = G_0(x_1, x_2, x), \quad (x_1, x_2, x, \mu + i\nu) \in L.
	\]
	Moreover, $G^*A_1 = A_2$. Finally, if $A_i$ are unitary for $i = 1, 2$, then $G_0$ takes values in $\mathrm{U}(r)$.
\end{lemma}
\begin{proof}
	\emph{Step 1: holomorphic family of $2$-frames.} We will differentiate suitably in the vertical fibres and show that these derivatives vanish. Fix $v \in S_xM_0$; following \cite{eskin-01} we consider the following family of complex vectors, for $t \in \mathbb{C}^\times := \mathbb{C} \setminus \{0\}$
	\begin{align}\label{eq:xi-def}
	\begin{split}
		&\xi^\pm(t) := \xi^\pm(x, v, t) := \frac{i}{2}\left(t - \frac{1}{t}\right) \partial_1 - \frac{1}{2}\left(t + \frac{1}{t}\right) \partial_2 \pm iv\\
		&= \underbrace{-\frac{1}{2}\im\left(t - \frac{1}{t}\right) \partial_1 - \frac{1}{2}\re\left(t + \frac{1}{t}\right) \partial_2}_{\mu(t) :=} + i\underbrace{\left(\frac{1}{2} \re\left(t - \frac{1}{t}\right) \partial_1 - \frac{1}{2} \im \left(t + \frac{1}{t}\right)\partial_2 \pm v\right)}_{\nu^\pm(t) :=}.
	\end{split}
	\end{align}
	Using complex bilinearity of $|\bullet|^2$, we see that $|\xi^\pm(t)|^2 = 0$; also, writing $t = a + i b$ with $a, b \in \mathbb{R}$, we compute
	\begin{align}\label{eq:small-computation}
	\begin{split}
		\im\left(t - \frac{1}{t}\right) &= b\left(1 + \frac{1}{a^2 + b^2}\right), \quad \re\left(t + \frac{1}{t}\right) = a\left(1 + \frac{1}{a^2 + b^2}\right),\\
		\re\left(t - \frac{1}{t}\right) &= a\left(1 - \frac{1}{a^2 + b^2}\right), \quad \im\left(t + \frac{1}{t}\right) = b\left(1 - \frac{1}{a^2 + b^2}\right),\\
		r &:= r(t) := |\mu(t)| = \left(1 + \frac{1}{a^2 + b^2}\right) \sqrt{a^2 + b^2}.
	\end{split}
	\end{align}
	By the last equality, $\mu(t) \neq 0 $ and so $\xi^\pm(t) \in L(x_1, x_2, x)$ for any $(x_1, x_2) \in \mathbb{R}^2$. Observe that the complexified geodesic vector field in direction $\xi^\pm(t)$ takes the form
	\begin{equation}\label{eq:complex-geodesic-vector-field-along-xi}
		\mathbb{X}(x_1, x_2, x, \xi^\pm(t)) = \frac{i}{2} \left(t - \frac{1}{t}\right)\partial_1 - \frac{1}{2} \left(t + \frac{1}{t}\right) \partial_2 \pm iX(x, v), \quad (x_1, x_2) \in \mathbb{R}^2.
	\end{equation}
	Using this formula, as well as linearity of $A_i$ for $i = 1, 2$, we get
	\[
		\mathbb{X}(\xi^\pm(t)) \partial_{\bar{t}}G  + A_1(\xi^\pm(t)) \cdot \partial_{\bar{t}} G - \partial_{\bar{t}} G \cdot A_2(\xi^\pm(t)) = 0, 
	\]
	where $\partial_{\bar{t}} G = 0$ in the component of $\partial (D \times M_0)$ of $D \times M_0 \setminus M$. For fixed $t \in \mathbb{C}^\times$, in the bicharacteristic leaf generated by $\xi^\pm(t)$, this equation takes the familiar form \eqref{eq:G-equation} (with $G$ replaced by $\partial_{\bar{t}}G$). Since the solution vanishes for large $|z|$, we conclude by uniqueness that $\partial_{\bar{t}} G \equiv 0$.
	
	We next look at the limits $t \to 0$ and $t \to \infty$ of $\xi^\pm(t)$. In fact, re-normalising and using \eqref{eq:small-computation}, and writing $A := \frac{a}{\sqrt{a^2 + b^2}}$ and $B := \frac{b}{\sqrt{a^2 + b^2}}$ we get
	\[
		\frac{\xi^\pm(t)}{r(t)} = -\frac{1}{2}B \partial_1 - \frac{1}{2} A \partial_2 + i\left(\frac{1}{2} \frac{a^2 + b^2 -1}{a^2 + b^2 + 1} \left(A  \partial_1 - B \partial_2\right) \pm \frac{v}{r(t)} \right).
	\]
	Therefore assuming that $A = A(t) \to 2y_1 \in \mathbb{R}$ and $B = B(t) \to 2y_2 \in \mathbb{R}$ converges in the limits $t \to 0$ or $t \to \infty$, the quotient becomes
	\[
		\frac{\xi^\pm(t)}{r(t)} \to_{t \to \infty} - y_2 \partial_1 - y_1 \partial_2 + i(y_1 \partial_1 - y_2 \partial_2), \quad \frac{\xi^\pm(t)}{r(t)} \to_{t \to 0} - y_2 \partial_1 - y_1 \partial_2 + i(-y_1 \partial_1 + y_2 \partial_2).
	\]
	When varied over the set of possible $(y_1, y_2)$, this is precisely the set of limit points of the quotient and note that according to Lemma \ref{lemma:symmetries} (Item 2), the value of $G$ is constant on these limit sets individually. Therefore, the function $\mathbb{C}^\times \ni t \mapsto G(\xi^\pm(t))$ is holomorphic and continuous close to zero and close to infinity. By the Removable Singularities Theorem, it then admits a holomorphic extension to $\mathbb{C}$, and by Liouville's theorem it is constant. We have thus shown that $G$ is constant on the set 
	\begin{align*}
		\mc{S}(x_1, x_2, x, v) &:= \left(\bigcup_{t \in \mathbb{C}^\times} \xi^+(x, v, t)\right) \cup \left(\bigcup_{t \in \mathbb{C}^\times} \xi^-(x, v, t)\right) \cup \{\partial_1 + i \partial_2\} \cup \{\partial_1 - i\partial_2\},\\
                \mc{S}(x_1, x_2, x) &:= \bigcup_{v \in S_xM_0} \mc{S}(x_1, x_2, x, v).
	\end{align*}
 \medskip
 
	\emph{Step 2: uniqueness.} We now set
	\begin{equation}\label{eq:gauge-def}
		G_0(x_1, x_2, x) := G(x_1, x_2, x, \partial_1 + i\partial_2), \quad (x_1, x_2, x) \in D \times M_0.
	\end{equation}
	
	We claim that $G_0$ satisfies the conditions of the lemma. Indeed, for $\mu + i\nu \in \mc{S}(x_1, x_2, x)$ and writing $\nu = a\mu_\perp + v$, where $\mu_\perp$ is the (oriented) unit normal to $\mu$, $a \in \mathbb{R}$, and $v \in T_xM_0$, we compute ($\gamma_v(t)$ is the geodesic generated by $(x, v)$):
	\begin{align}\label{eq:complex-gauge-equivalence}
    \begin{split}
		&dG_0(x_1, x_2, x)(\mu + i\nu) = \partial_t|_{t = 0} G_0\big((x_1, x_2) + t\mu, x)\big) + i\partial_t|_{t = 0} G_0\big((x_1, x_2) + t a\mu_\perp, \gamma_v(t)\big)\\
		&= \partial_t|_{t = 0} G\big((x_1, x_2) + t\mu, x, \mu + i\nu\big) + i\partial_t|_{t = 0} G\big((x_1, x_2) + t a\mu_\perp, \gamma_v(t), \mu + i(a\mu_\perp + \dot{\gamma}_v(t)\big)\\
		&= \mathbb{X}(x_1, x_2, x)(\mu + i\nu) G(x_1, x_2, x, \mu + i \nu)\\
		&= -A_1(x_1, x_2, x)(\mu + i\nu) G_0(x_1, x_2, x) + G_0(x_1, x_2, x) A_2(x_1, x_2, x),
    \end{split}
	\end{align}
	where in the second equality we used the definition of $G_0$ as well as the fact that $\mc{S}$ is invariant under the geodesic flow, in the third equality we used the definition of the complexified geodesic vector field $\mathbb{X}$, and in the last equality we used \eqref{eq:homomorphism-connection-complex-transport}. Plugging in $\xi^\pm(x, v, t)$ into the preceding equality (for an arbitrary fixed $t \in \mathbb{C}^\times$) and subtracting we obtain
	\begin{equation}\label{eq:v-gauge-equivalence}
		G_0^{-1} dG_0(v)  + G_0^{-1} A_1(v) G_0 = A_2(v), \quad v \in T_x M_0.
	\end{equation}
	Furthermore, since $\partial_1 \pm i \partial_2 \in \mc{S}(x_1, x_2, x)$, subtracting and summing, and combining with \eqref{eq:v-gauge-equivalence}, we conclude that $G_0^* A_1 = A_2$. Finally, we are left to observe that hence $G_0$ also satisfies \eqref{eq:homomorphism-connection-complex-transport} and so by uniqueness of solutions $G$ is independent of the $2$-frame variable, completing the proof in the $\GL_r(\mathbb{C})$ case.

    In the unitary case, it suffices to observe that the gauge defined in \eqref{eq:gauge-def} is unitary. Indeed, we have
	\[
		G_0 = G(\partial_1 + i\partial_2) = G^{-*}(\partial_1 - i\partial_2) = G^{-*}(\partial_1 + i\partial_2) = G_0^{-*},
	\]
	where in the second equality we used Lemma \ref{lemma:symmetries} (Item 3) and in the third one we used that $G_0$ is constant on $\mc{S}(x_1, x_2, x)$, concluding the proof.
\end{proof}

Since we only know that $G_0$ from the preceding lemma is equal to identity on the outer boundary, we make a reduction of our problem.

\begin{lemma}\label{lemma:convexification}
	There exists a compact manifold with boundary $\widetilde{M} \subset \mathbb{R}^2 \times M_0$, containing $M$, such that $\partial \widetilde{M}$ is connected, and $0$ is not a Dirichlet eigenvalue of the operators $\Lie_{A_i, Q_i}$ over $\widetilde{M}$ for $i = 1, 2$. If $\Lambda_{A_1, Q_1} = \Lambda_{A_2, Q_2}$, then the DN maps $\widetilde{\Lambda}_{A_1, Q_1}$ and $\widetilde{\Lambda}_{A_2, Q_2}$ defined over $\widetilde{M}$ agree, $\widetilde{\Lambda}_{A_1, Q_1} = \widetilde{\Lambda}_{A_2, Q_2}$.
\end{lemma}
\begin{proof}
	By Lemma \ref{lemma:topology-1}, there exists a compact smooth manifold $M' \subset \mathbb{R}^2 \times M_0$ with boundary equal to the outer boundary of $M$. Then by possibly increasing $M'$ to $\widetilde{M}$, and using \emph{domain monotonicity} (i.e. monotonicity of Dirichlet eigenvalues with respect to enlarging the domain), we can guarantee that $0$ is not a Dirichlet eigenvalue of neither $\Lie_{A_1, Q_1}$ nor $\Lie_{A_2, Q_2}$ over $\widetilde{M}$. Therefore, the Dirichlet problem is well-posed and uniquely solvable, and in particular the DN maps $\widetilde{\Lambda}_{A_1, Q_1}$ and $\widetilde{\Lambda}_{A_2, Q_2}$ are well-defined.
	
	For the final claim, consider arbitrary $\widetilde{F}_1, \widetilde{F}_2 \in C^\infty(\partial \widetilde{M}, \mathbb{C}^{r \times r})$; there are unique $\widetilde{U}_1, \widetilde{U}_2 \in C^\infty(\widetilde{M}, \mathbb{C}^{r \times r})$ satisfying
	\[
		\Lie_{A_i, Q_i} \widetilde{U}_i = 0, \quad \widetilde{U}_i|_{\partial \widetilde{M}} = \widetilde{F}_i, \quad i = 1, 2.
	\]
	Then we compute
	\begin{align*}
		&\langle{(\widetilde{\Lambda}_{A_1, Q_1} - \widetilde{\Lambda}_{A_2, Q_2}) \widetilde{F}_1, \widetilde{F}_2}\rangle_{L^2(\partial \widetilde{M}, \mathbb{C}^{r \times r})}\\ 
		&=  \langle{(Q_1 - Q_2 + |A_2|^2 - |A_1|^2) \widetilde{U}_1, \widetilde{U}_2}\rangle_{L^2(\widetilde{M}, \mathbb{C}^{r \times r})} + \langle{\widetilde{U}_1\cdot d\widetilde{U}_2^* - d\widetilde{U}_1\cdot \widetilde{U}_2^*, A_1^* - A_2^*}\rangle_{L^2(\widetilde{M}, \mathbb{C}^{r \times r})}\\
		&= \langle{(Q_1 - Q_2 + |A_2|^2 - |A_1|^2) \widetilde{U}_1, \widetilde{U}_2}\rangle_{L^2(M, \mathbb{C}^{r \times r})} + \langle{\widetilde{U}_1\cdot d\widetilde{U}_2^* - d\widetilde{U}_1\cdot \widetilde{U}_2^*, A_1^* - A_2^*}\rangle_{L^2(M, \mathbb{C}^{r \times r})}\\
		&= \langle{(\Lambda_{A_1, Q_1} - \Lambda_{A_2, Q_2}) F_1, F_2}\rangle_{L^2(\partial M, \mathbb{C}^{r \times r})} = 0,
	\end{align*}
	where in the first equality we used Proposition \ref{prop:integral-identity}, in the second one we used that $A_2 = A_1$ and $Q_2 = Q_1$ on $\widetilde{M} \setminus M$, and in the final line we used Proposition \ref{prop:integral-identity} again and the equality $\Lambda_{A_1, Q_1} = \Lambda_{A_2, Q_2}$. Since $\widetilde{F}_1$ and $\widetilde{F}_2$ were arbitrary, this gives $\widetilde{\Lambda}_{A_1, Q_1} \equiv \widetilde{\Lambda}_{A_2, Q_2}$, which completes the proof.
\end{proof}

Next, we deal with the potentials and we show in fact that the gauge constructed in Lemma \ref{lemma:connection-uniqueness} intertwines the potentials as well.

\begin{lemma}\label{lemma:potential-uniqueness}
	If $G := G_0$ is the automorphism constructed in Lemma \ref{lemma:connection-uniqueness}, then $G^*Q_1 = Q_2$.
\end{lemma}
\begin{proof}
 	Let $\widetilde{M}$ be the compact manifold with boundary constructed in Lemma \ref{lemma:convexification}. Observe that by Lemma \ref{lemma:connection-uniqueness}, $G|_{\partial \widetilde{M}} = \id$ and so $\widetilde{\Lambda}_{G^*A_1, G^*Q_1} = \widetilde{\Lambda}_{A_2, Q_2}$; write $A := G^*A_1 = A_2$ and $Q := G^*Q_1$. Then we would like to prove $Q \equiv Q_2$. By Theorem \ref{thm:reduction-to-complex-x-ray-potential} (applied to $\widetilde{M}$), there exists $F \in C^\infty(L, \mathbb{C}^{r \times r})$ such that
 	\begin{equation}\label{eq:F-equation-convexified}
 		(\mathbb{X}(\mu + i\nu) + [A(\mu + i\nu), \bullet])F = Q - Q_2, \quad F|_{\mathbb{R}^2 \times M_0 \setminus M} = 0.
 	\end{equation}
 	Note that $F|_{\mathbb{R}^2 \times M_0 \setminus M} = 0$ follows from $F|_{\mathbb{R}^2 \times M_0 \setminus \widetilde{M}} = 0$ and the unique continuation principle. We would like to show that $F \equiv 0$; we will use the same technique and notation as in Lemma \ref{lemma:connection-uniqueness}. Using the family $\xi^\pm(t) \in L$ defined in \eqref{eq:xi-def} and differentiating with $\partial_{\bar{t}}$ we get
 	\[
 		(\mathbb{X}(\xi^\pm(t)) + [A(\xi^\pm(t)), \bullet]) \partial_{\bar{t}} F = 0, \quad \partial_{\bar{t}} F|_{\mathbb{R}^2 \times M_0 \setminus M} = 0,
 	\]
 	where the commutation of $\mathbb{X}$ and $\partial_{\bar{t}}$ is justified using \eqref{eq:complex-geodesic-vector-field-along-xi}, and we used linearity of $A$ in its argument. For each fixed $t \in \mathbb{C}^\times$, by uniqueness of solutions to $(\partial_{\bar{z}} + [A(\partial_{\bar{z}}), \bullet]) \partial_{\bar{t}} F = 0$ in the bicharacteristic leaves generated by $\xi^\pm(t)$, we get that $\partial_{\bar{t}} F \equiv 0$.
 	
 	Next, for any fixed $(x_1, x_2, x)$ and $v \in S_xM$, we claim that $F(x_1, x_2, x)(\xi^\pm(t)) \to 0$ as $t \to 0$ or $t \to \infty$. Indeed, by $-1$-homogeneity we have
 	\[
 		F(\xi^\pm(t)) = \frac{1}{|\mu(t)|} F\left(\frac{\xi^\pm(t)}{|\mu(t)|}\right) \to 0, \quad t \to 0, \infty, 
 	\]
 	since $F|_{L_1}$ is bounded and as $|\mu(t)| \to \infty$ as $t \to 0, \infty$ by the last line of \eqref{eq:small-computation}. Therefore by the Removable Singularities Theorem, and by Liouville's theorem, we get $F(\xi^\pm(t)) \equiv 0$. By the fact $(x_1, x_2, x)$ was arbitrary, and since $\mc{S}$ is invaraint by the geodesic flow, coming back to \eqref{eq:F-equation-convexified} yields $Q \equiv Q_2$ which completes the proof of the lemma. In fact, a posteriori also by uniqueness in \eqref{eq:F-equation-convexified} we get $F \equiv 0$, which proves the claim.
\end{proof}

We are in good shape to prove one of the main theorems.

\begin{proof}[Proof of Theorem \ref{thm:main-theorem-simple}]
We divide the proof in two steps.
\medskip

\emph{Step 1: case $c \equiv 1$.} By Lemmas \ref{lemma:connection-uniqueness}, \ref{lemma:convexification}, and \ref{lemma:potential-uniqueness}, there exists $G \in C^\infty(M, \GL_r(\mathbb{C}))$ such that $G^*A_1 = A_2$ and $G^*Q_1 = Q_2$, and $G = \id$ on the outer boundary $\mc{O}$ of $M$. We are therefore left to show that $G|_{\partial M} = \id$, i.e. $G$ is equal to the identity on the whole of $\partial M$. Indeed, let $F \in C^\infty(\partial M, \mathbb{C}^{r \times r})$ be arbitrary, and let $U_1, U_2 \in C^\infty(M, \mathbb{C}^{r \times r})$ be the unique solutions to 
	\[
		\Lie_{A_i, Q_i} U_i = 0, \quad U_i|_{\partial M} = F, \quad i = 1, 2.
	\]
	Notice that by assumption $G^{-1} \Lie_{A_1, Q_1} G = \Lie_{A_2, Q_2}$; thus $V := G U_2$ satisfies $\Lie_{A_1, Q_1}V = 0$ and $V|_{\mc{O}} = U_1|_{\mc{O}}$. Moreover, we have
	\[
		\partial_\nu V|_{\mc{O}} = (\partial_\nu G \cdot U_2 + G \cdot \partial_\nu U_2)|_{\mc{O}} = \big(G A_2(\partial_\nu) - A_1(\partial_\nu) G\big) U_2 |_{\mc{O}} + \partial_\nu U_2|_{\mc{O}} = \partial_{\nu} U_2|_{\mc{O}},
	\]
	where in the second equality we used $G^*A_1 = G^{-1}dG + G^{-1} A_1 G = A_2$, as well as that $G|_{\mc{O}} = \id$, and in the third equality we used again $G|_{\mc{O}} = \id$ and that $A_1$ and $A_2$ agree on the boundary. Therefore, by unique continuation for diagonal elliptic systems (see \cite[Theorem 3.5.2]{isakov-17}), we conclude that $V \equiv U_2$ and in particular 
	\[
		G (x) F(x) = G(x) U_1(x) = V(x) = U_2(x) = F(x), \quad x \in \partial M.
	\]
	Since $F$ was arbitrary, we conclude that $G|_{\partial M} \equiv \id$, which completes the proof in the case $c \equiv 1$.
	\medskip
	
	\emph{Step 2: arbitrary $c$.} Writing $\Lie_{g, A, Q} = d_A^* d_A + Q$ (where we now emphasise the dependence on the metric $g$; $A$ is a connection and $Q$ a potential), we then have the identity
	\begin{equation}\label{eq:conformal-identity}
		c^{\frac{n + 2}{4}}\Lie_{c g, A, Q} c^{-\frac{n - 2}{4}} = \Lie_{g, A, c(Q + Q_c)}, \quad Q_c := c^{\frac{n - 2}{4}} (-\Delta_g) (c^{-\frac{n - 2}{4}}).
	\end{equation}
	We will revisit Theorems \ref{thm:reduction-to-complex-x-ray-connection} and \ref{thm:reduction-to-complex-x-ray-potential} and the construction of CGOs. We consider first Theorem \ref{thm:reduction-to-complex-x-ray-connection} and freely use its notation; we will highlight the main differences to the case $c \equiv 1$. By \eqref{eq:conformal-identity} we may consider solutions
	\[
		U_i = e^{-\frac{x_\mu + ir}{h}} c^{-\frac{n - 2}{4}} (|g|^{-\frac{1}{4}} C_i + R_i), \quad i = 1, 2,
	\]
	where $C_1$ and $C_2$ satisfy \eqref{eq:transport-1-simple} and \eqref{eq:transport-2-simple}, respectively (as before). In the integral identity \eqref{eq:integral-identity-inside}, the term $T_1$ in \eqref{eq:T-1} remains equal to $\mc{O}(1)$ since $c$ is bounded. For the term $T_2$ in \eqref{eq:T-2}, we note that the term involving $dc$ is equal to $\mc{O}(1)$, so we get a multiplicative factor of the form
	\begin{align*}
		T_2 &= \mc{O}(1) + 2h^{-1} \int_M \Tr \left(C_1 C_2^* (\widetilde{A}_1 + i\widetilde{A}_r)\right)\, c^{-\frac{n - 2}{2}} c^{-1} c^{\frac{n}{2}} |g|^{-\frac{1}{2}}\, d\vol_{g}\\
		&= \mc{O}(1) + 2h^{-1} \int_M \Tr \left(C_1 C_2^* (\widetilde{A}_1 + i\widetilde{A}_r)\right)\, dx_\mu dr d\theta,
	\end{align*}
	where in the first line we used that $d\vol_{cg} = c^{\frac{n}{2}} d\vol_g$, and the factor $c^{-1}$ appears by using the formula $\langle{\bullet, \bullet}\rangle_{cg} = c^{-1} \langle{\bullet, \bullet}\rangle_g$. Hence we obtain the same formula and the construction of the gauge $G$ such that $G^*A_1 = A_2$ remains the same then on (i.e. the remainder of Theorem \ref{thm:reduction-to-complex-x-ray-connection} and Lemma \ref{lemma:connection-uniqueness}).
	
	For the uniqueness of potential, we look again at Theorem \ref{thm:reduction-to-complex-x-ray-potential} and the integral identity \eqref{eq:integral-identity-inside-potentials}. We get instead in the limit $h \to 0$:
	\[
		\int_{B(\nu)} C^{-1} c(Q_1 - Q_2) C\, dx_\mu dr d\theta = 0,
	\]
	i.e. there is an additional multiplicative factor $c$. From then on (in the remainder of Theorem \ref{thm:reduction-to-complex-x-ray-potential} and Lemma \ref{lemma:potential-uniqueness}) we see it is possible to work with $cQ_i$ instead of $Q_i$, for $i = 1, 2$. Since $c > 0$, we conclude that $Q_1 \equiv Q_2$, which completes the proof.	 
\end{proof}

\section{Uniqueness in the setting of arbitrary transversal geometry}\label{sec:gaussian-beams}

In this section we prove Theorem \ref{thm:main-theorem-gaussian-beams}, allowing the transversal manifold $(M_0, g_0)$ to be an arbitrary Riemannian manifold. The proof relies on the Gaussian Beams construction carried out in \cite{cekic-17}, which replaces the simplicity assumption on $(M_0, g_0)$. However, the main ideas have already been laid out in the previous two sections, so we will mainly focus on the differences in the proofs. By Lemma \ref{lemma:topology-2} we may assume without loss of generality that $(M_0, g_0)$ has strictly convex and connected boundary. We consider the extension $(N_0, g_0)$ of $(M_0, g_0)$ as constructed in \S \ref{ssec:extensions}, as well as an intermediate extension $(M_{0e}, g_0)$ with strictly convex boundary.

Say that a geodesic segment between boundary points $\gamma: [0, L] \to M_0$ is \emph{non-tangential} if $\dot{\gamma}(0)$ and $\dot{\gamma}(L)$ are transversal to the boundary $\partial M_0$ and $\gamma((0, L)) \subset M_0^\circ$. Define the set of \emph{non-trapped bicharacteristics} $\mc{L}_1 \subset L_1$
\[
	\mc{L}_1 := \{(x_1, x_2, x, \mu + i\nu) \in L_1 \mid d\pi_{\mathbb{R}^2}(x_1, x_2, x) \nu \neq 0 \,\, \mathrm{or}\,\, \nu \in S_xM_0\,\,\mathrm{and}\,\,\gamma_\nu \,\, \mathrm{is} \,\, \mathrm{non\text{-}tangential}\},
\]
where $\pi_{\mathbb{R}^2}$ is the projection onto $\mathbb{R}^2$ and $\gamma_\nu$ is the geodesic generated by $\nu$. The idea is that for non-trapped bicharacteristics we can solve the Del Bar equation and construct suitable CGOs. Observe that $\mc{L}_1$ is open: indeed, the condition $d\pi_{\mathbb{R}^2}(x_1, x_2, x) \nu \neq 0$ is trivially open, while $\nu$ non-tangential is open by strict convexity of $\partial M_0$. Observe also that $\mc{L}_1$ is invariant under the action of the geodesic flow, that is, it is a union of bicharacteristic leaves. We will also need the scaled set of non-trapped bicharacteristics
\[
	\mc{L} := \{(x_1, x_2, x, r(\mu + i\nu)) \in \mc{L}_1 \mid r \in \mathbb{R}_{> 0} \} \subset L.
\]

\begin{proof}[Proof of Theorem \ref{thm:main-theorem-gaussian-beams}]
	For simplicity we consider the case $c \equiv 1$; the case $c \not \equiv 1$ follows by similar considerations as in the proof of Theorem \ref{thm:main-theorem-simple}. We proceed in steps, giving analogues of Theorems \ref{thm:reduction-to-complex-x-ray-connection} and \ref{thm:reduction-to-complex-x-ray-potential} along the way.
	
	\medskip
	
	\emph{Step 1: preparation.} We claim that Theorem \ref{thm:reduction-to-complex-x-ray-connection} still holds, with the only change being that \eqref{eq:homomorphism-connection-complex-transport} holds on the set of non-trapped bicharacteristics, that is, there is a $G \in C^\infty(\mc{L}_1|_{\mathbb{R}^2 \times M_0}; \GL_r(\mathbb{C}))$ satisfying \eqref{eq:homomorphism-connection-complex-transport}, with $G \equiv \id$ on the connected component of $\mathbb{R}^2 \times \partial M_0$ inside $\mathbb{R}^2 \times M_0 \setminus M$. We give a slightly simplified proof, as we do not have to construct a family of solutions to \eqref{eq:transport-1-simple} for points close to the boundary; we argue directly with points of $\mc{L}_1$.
	
	For any $(x_1, x_2, x, \mu + i\nu) \in \mc{L}_1|_{\mathbb{R}^2 \times M_{0e}^\circ}$, consider the bicharacteristic leaf $B(\mu + i\nu)$ it generates and the natural coordinates $z = s + it \in \mathbb{C}$ (as in \eqref{eq:leaf-explicit}); $B(\mu + i \nu)$ is then identified with $\mathbb{C}$. Let $\mc{U} \subset \mc{L}_1$ be an open neighbourhood of $(x_1, x_2, x, \mu + i \nu)$ diffeomorphic to a disk. Consider the equation
	\[
		\partial_{\bar{z}} C_1 + A_1(\partial_{\bar{z}}) C_1 = 0, \quad z \in B(\mu + i\nu) \simeq \mathbb{C}.
	\]
	By the non-trapping assumption, $A_1(z)$ has compact support in $z \in \mathbb{C}$. Similarly to Lemma \ref{lemma:parametric-del-bar}, we may solve this equation with  $C_1 \in C^\infty(\mc{U} \times \mathbb{C}, \GL_r(\mathbb{C}))$ and
	\[
		\partial_z C_2 - A_2^*(\partial_z) C_2 = 0, \quad z \in B(\mu + i\nu), \quad C_2 \in C^\infty(\mc{U} \times \mathbb{C}; \GL_r(\mathbb{C})).
	\]
	Now, form $G := C_1 C_2^*$; it satisfies the same equation as \eqref{eq:G-equation}. If we were able to modify $C_2$ as in the proof of Theorem \ref{thm:reduction-to-complex-x-ray-connection} (Step 2), so that $G(z) \equiv \id$ for large enough $z$, then by uniqueness of solutions we would be able to glue $G$ over a family of open sets $\mc{U}$ and to prove the claim. Moreover, we claim that the symmetries obtained in Lemma \ref{lemma:symmetries} are still valid. This is all proved in the next step.
	
	\medskip
	
	\emph{Step 2: application of CGOs.} Without loss of generality we may assume that $\mu = \e_1$ (just apply a rotation in $\mathbb{R}^2$ that takes $\mu$ to $\e_1$). Then consider the (bi-infinite) geodesic $\gamma_\nu$ generated by $\nu$. We claim there exists a smooth compact manifold with boundary $\mc{T}_0 \Subset \mathbb{R} \times N_0^\circ$ such that a segment of $\gamma_\nu$ is non-tangential in $\mc{T}_0$ and $M \subset \mathbb{R} \times \mc{T}$. Indeed, if the component in the direction of the second component of $\mathbb{R}$ of $\nu$ is non-zero, then considering $[-T, T] \times N_1$ for some large $T$, and taking a slightly smaller $N_1 \Subset N_0^\circ$, $\gamma_\nu$ intersects $x_2 = T$ and $x_2 = -T$ transversely; smoothing out the corners will work. If $\nu \in S_xM$, then $\gamma_\nu$ is a non-tangential geodesic in $M_0$, then again taking $\mc{T}$ to be $[-T, T] \times M_0$ for large $T$ and smoothing out corners again will work; this time $\gamma_\nu$ is transversal to $\mathbb{R} \times \partial M_0$.

We our now in shape to apply \cite[Theorem 5.8 and Proposition 5.10]{cekic-17} with $\lambda = 0$ and $\mathbb{R} \times \mc{T}_0$. Thus, there are solutions $U_i$ for $i = 1, 2$, satifying $\Lie_{A_1, Q_1} U_1 = 0$ and $\Lie_{-A_2^*, Q_2^*} U_2 = 0$, of the form
	\begin{equation}\label{eq:gaussian-beams-form}
		U_1 = e^{h^{-1} x_1} (V_1 + R_1), \quad U_2 = e^{-h^{-1} x_1} (V_2 + R_2),
	\end{equation}
	where $h > 0$ is a small parameter, such that as $h \to 0$:
	\begin{equation}\label{eq:gaussian-beams-decay}
		\|R_i\|_{L^2} = o(1), \quad \|R_i\|_{H^1} = o(h^{-1}), \quad \|V_i\|_{L^2} = \mc{O}(1), \quad i = 1, 2,
	\end{equation}
	and the following concentration properties hold, writing $\gamma := \gamma_\nu$:
	\begin{align}\label{eq:gaussian-beams-concentration}
	\begin{split}
		\lim_{h \to 0} \int_{\{x_0\} \times \mc{T}_0} \Tr\big(V_1 V_2^*\big) \varphi \, d\vol_g &= \int_0^L \Tr(C_1 C_2^*) \varphi(\gamma(t))\, dt\\
		\lim_{h \to 0} h \int_{\{x_0\} \times \mc{T}_0} \Tr\big(\langle{\alpha, dV_2^*} \rangle V_1\big) \varphi \, d\vol_g &= -i\int_0^L \Tr(C_1 C_2^*) \varphi(\gamma(t))\, dt\\
		\lim_{h \to 0} h \int_{\{x_0\} \times \mc{T}_0} \Tr\big(\langle{\alpha, dV_1} \rangle V_2^*\big) \varphi \, d\vol_g &= i\int_0^L \Tr(C_1 C_2^*) \varphi(\gamma(t))\, dt,
	\end{split}
	\end{align}
	where $x_0 \in \mathbb{R}$ and $\varphi \in C^\infty(\{x_0\} \times \mc{T}_0)$ are arbitrary. Using the integral identity as in \eqref{eq:integral-identity-inside} with the solutions \eqref{eq:gaussian-beams-form}, we get two terms $T_1$ and $T_2$ as in \eqref{eq:T-1} and \eqref{eq:T-2}, respectively. By \eqref{eq:gaussian-beams-decay}, we again have $T_1 = \mc{O}(1)$ as $h \to 0$. For the term $T_2$, we get
	\begin{align*}
		T_2 &= \int_M \Tr\big(\langle{U_1 \cdot dU_2^*-  dU_1 \cdot U_2^*, A_1^* - A_2^*}\rangle\big)\, d\vol_g\\
		&= \mc{O}(1) + 2\int_M \Tr\left(h^{-1} (V_1 + R_1) (V_2^* + R_2^*) \widetilde{A}_1\right)\, d\vol_g\\ 
		&- \int_M \Tr \left(\langle{V_1 \cdot dV_2^* - dV_1 \cdot V_2^*, \widetilde{A}}\rangle\right) \, d\vol_g,
	\end{align*}
	where we recall $\widetilde{A} = A_2 - A_1$, and in the second line we used \eqref{eq:gaussian-beams-decay} to bound the remaining factors. Therefore we obtain
	\begin{align*}
		0 = \lim_{h \to 0} h(T_1 + T_2) &= 2 \int_{B(\mu + i\nu)} \Tr\big(C_1 C_2^* \widetilde{A}_1\big)\, dx_1 dt + 2i \int_{B(\mu + i\nu)} \Tr\big(C_1 C_2^* \widetilde{A}(\dot{\gamma}(t))\big)\, dx_1 dt\\
		&= 2\int_{B(\mu + i\nu)} \Tr\big(C_1 C_2^* (\widetilde{A}_1 + i \widetilde{A}_r)\big)\, dx_1 dr,
	\end{align*}
	where we we used the concentration properties \eqref{eq:gaussian-beams-concentration} in the first line, and wrote $r$ for the coordinate in the direction of $\gamma$ in the bicharacteristic leaf $B(\mu + i\nu)$. Notice that this is verbatim the same identity as in \eqref{eq:CGO-limit-inside-connections}, and so the remainder of the proof in Theorem \ref{thm:reduction-to-complex-x-ray-connection}, Steps 2 and 4 (Step 3 is no longer required), applies to give a solution $G$ as claimed in Step 1 (extending by $0$-homogeneity we obtain a solution over $\mc{L}$). It is left to notice that the symmetries in Lemma \ref{lemma:symmetries} are still valid in the same way as before, because $\mc{L}$ is invariant under the mentioned symmetries.
 \medskip

 \emph{Step 3: uniqueness for the connection.} Using the notation of Lemma \ref{lemma:connection-uniqueness}, we notice that $\xi^\pm(t) \in \mc{L}$ for all $t \in \mathbb{C}^\times$ if and only if $v \in S_xM$ is non-tangential. Indeed, by \eqref{eq:small-computation} we see that $d\pi_{\mathbb{R}^2} \nu^\pm(t) \neq 0$ if and only if $t \not \in \mathbb{S}^1$; and $\nu^\pm(t) = \pm v$ for $t \in \mathbb{S}^1$. So for non-tangential $v$ the same proof gives $\partial_{\bar{t}} G \equiv 0$; in fact, by the argument involving Liouville's theorem we get that $G$ is constant on the set $\mc{S}$ if $v$ is non-tangential. We define $G_0 := G(\partial_1 + i\partial_2) \in C^\infty(\mathbb{R}^2 \times M_0, \GL_r(\mathbb{C}))$ as before and show it satisfies $G_0^*A_1 = A_2$. Indeed, it satisfies \eqref{eq:complex-gauge-equivalence} in the same way, and \eqref{eq:v-gauge-equivalence} also holds for all non-tangential $v$. In fact, by openness of the set of non-tangential directions, and by linearity, once \eqref{eq:v-gauge-equivalence} holds for one such $v$, it holds for all $v \in T_xM$. Moreover, as before we get $G_0^*A_1(\partial_i) = A_2(\partial_i)$ for $i = 1, 2$, and thus $G_0^*A_1 = A_2$ on the set $\mathbb{R}^2 \times \Omega$, where
\[
    \Omega := \{x \in M_0 \mid \exists v \in S_xM_0 \,\,\mathrm{non\text{-}tangential}\},
\]
 with $G_0 = \id$ on the outer boundary of $M$. It is left to observe that the set $\Omega$ is open (by strict convexity of $\partial M_0$) and dense since it is of measure zero by \cite[Lemma 3.1]{salo-17}. Since $G_0^*A_1 = A_2$ over $\mathbb{R}^2 \times \Omega$, and both quantities extend smoothly to $\mathbb{R}^2 \times M_0$, by continuity we conclude that $G_0^* A_1 \equiv A_2$. Finally, if $A_1$ and $A_2$ are unitary, symmetries of $G_0$ give that $G_0$ takes values in $\mathrm{U}(r)$, completing the proof.
 \medskip

 \emph{Step 4: uniqueness for the potential.} It is now clear how to proceed to show that if $G := G_0$, then $G^*Q_1 = Q_2$. Indeed, one first constructs the map $F$ solving the complex transport equation as in Theorem \ref{thm:reduction-to-complex-x-ray-potential} for non-trapped bicharacteristics $\mc{L}$; this is done similarly to Steps 1 and 2 above and we skip the details. Then, one argues as in Lemma \ref{lemma:potential-uniqueness} and Step 3 above; again, we skip the details (now, the manifold $\widetilde{M}$ used in Lemma \ref{lemma:convexification} is constructed in Lemma \ref{lemma:topology-2}). 
 
 To show $G|_{\partial M} \equiv \id$, it is left to argue as in the proof of Theorem \ref{thm:main-theorem-simple}, Step 1, where instead of Lemma \ref{lemma:topology-1} we use Lemma \ref{lemma:topology-2}. This proves the main theorem. 
\end{proof}

\section{Complex ray transform and complex parallel transport}\label{sec:complex-ray-transform}

In this section we propose a new \emph{complex ray transform} and a new \emph{complex parallel transport} problem based on integration on bicharacteristic leaves, and study its basic properties. 

\subsection{Complex ray transform} For the ease of presentation we assume $(M, g)$ is a simple manifold and we let $r \in \mathbb{Z}_{\geq 1}$; let $(N, g)$ be the extension of $(M, g)$ satisfying properties as in \S \ref{ssec:extensions}. Denote by $\partial_- SM$ the set of \emph{inward} pointing vectors, i.e.
\[
    \partial_-SM := \{(x, v) \in \partial SM \mid g_x(v, \nu) < 0\}.
\]
where $\nu$ is the outer normal to $\partial M$. Given $(x, v) \in \partial_- SM$, denote the bicharacteristic leaf at $(x, v)$ as
\[
    B(x, v) := B(v) := \mathbb{R} \times \gamma,
\]
where $\gamma :(-\infty, \infty) \to N$ is the geodesic generated by $(x, v)$ (more precisely, on the complement of $M$ it is the re-parametrisation of this geodesic, see \S \ref{ssec:extensions}), see Figure \ref{fig:complex-ray}. In what follows we will denote by $x_1$ and $r$ the $\mathbb{R}$ and the geodesic coordinate on $\gamma$, respectively, and by $z = x_1 + ir$ the complex coordinate on $B(x, v) \simeq \mathbb{C}$; recall $\partial_{\bar{z}} = \frac{1}{2}(\partial_{x_1} + i \partial_r)$. Given $m \in \mathbb{Z}_{\geq 0}$, denote by $\otimes_S^m T^*(\bullet)$ the vector bundle of \emph{symmetric $m$-tensors} on the manifold $\bullet$.

\begin{definition}\label{def:complex-ray}
    Let $T > 0$, $m \in \mathbb{Z}_{\geq 0}$ and let $A$ denote a smooth connection with compact support in $(-T, T) \times M^\circ$. Let $f \in C_{\comp}^\infty((-T, T) \times M^\circ, \otimes_S^m T^*(\mathbb{R} \times M^\circ) \otimes \mathbb{C}^{r \times r})$. The (\emph{non-Abelian}) \emph{complex ray transform} of $f$ at $(x, v) \in \partial_-SM$ is defined as
    \[
        \mc{C}f(x, v) := u_f|_{\partial([-T, T] \times [0, L])},    
    \]
    where $\gamma_{x, v}: [0, L] \to M$ is the maximal geodesic generated by $(x, v)$ and $u_f \in L^2_\tau(\mathbb{C}) \cap C^\infty(\mathbb{C})$ is the unique solution to the equation (given by Lemma \ref{lemma:del-bar-source})
    \begin{equation}\label{eq:del-bar-x-ray}
        \partial_{\bar{z}} u_f + A(z)(\partial_{\bar{z}}) u_f = f(z)(\otimes^m \partial_{\bar{z}}), \quad z = x_1 + ir, \quad z \in B(x,v) \simeq \mathbb{C}.
    \end{equation}
\end{definition}

When we want to emphasise the dependence on $m$ or the connection $A$ we will sometimes write $\mc{C}_m$ or $\mc{C}_A$, respectively, for the corresponding complex ray transform. Since $A$ and $f$ have compact supports, it is clear that $\mc{C}f$ does not depend on the extension $(N, g)$. Note that we restricted our attention to tensors since these have a natural extension to complex vectors such as $\partial_{\bar{z}}$; more generally, we could start with an arbitrary $f$ on $(-T, T) \times SM^\circ$ and then identify $f(z)$ with $f(x_1, x, v)$. Another reason to consider tensors is that the complex ray transform in the current form appears naturally in the \emph{non-linear} problem in \cite[Question 1.6]{cekic-17}, see also \S \ref{ssec:complex-parallel-transport} below.

Perhaps a more invariant way to define $\mc{C}$ is to use the holomorphic structure and consider restrictions to complements of large disks of $u_f$. For $d > 0$ denote the open disk of radius $d$ by $\mathbb{D}(d) \subset \mathbb{C}$ and by $\mathbb{D}^\times := \mathbb{D}(1) \setminus \{0\}$ the punctured unit disk at zero; also, denote the space of holomorphic matrix valued functions over an open set $\Omega \subset \mathbb{C}$ as $\mc{H}(\Omega, \mathbb{C}^{r \times r})$.

\begin{lemma}\label{lemma:complex-x-ray-equivalence}
    In the notation of Definition \ref{def:complex-ray}, there exists $d > 0$ large enough (depending only on $T$ and the diameter of $(M, g)$), such that for any $(x, v) \in \partial_- SM$, $\mc{C}f(x, v)$ is determined (and vice versa) by 
    \[
        \widetilde{\mc{C}}f (x, v) := u_f|_{\mathbb{C} \setminus \overline{\mathbb{D}}(d)} \circ T_d \in \mc{H}(\mathbb{D}^\times, \mathbb{C}^{r \times r}),
    \]
    where $T_d : \mathbb{D}^\times \to \mathbb{C} \setminus \overline{\mathbb{D}}(d)$ is the scaling map $T_d(z) := \frac{d}{z}$. Therefore, the complex ray transform can be seen as a map
    \[
        \widetilde{\mc{C}}: C_{\comp}^\infty((-T, T) \times M, \otimes_S^m T^*(\mathbb{R} \times M) \otimes \mathbb{C}^{r \times r}) \to C^\infty(\partial_{-}SM, \mc{H}(\mathbb{D}^\times, \mathbb{C}^{r \times r})).
    \]
\end{lemma}
\begin{proof}
    Note that $u_f$ is holomorphic in the $z$-variable outside of the supports of $A$ and $f$. It therefore suffices to show that a holomorphic function $h$ is uniquely determined by its values on a smooth curve; however, this is obvious as holomorphic functions have a discrete set of zeroes.
\end{proof}

Next, we solve the transport problem associated to the complex ray transform.

\begin{lemma}\label{lemma:complex-transport-reduction}
    For any $f$ as in Definition \ref{def:complex-ray}, there exists $u \in C^\infty([-T, T] \times SM, \mathbb{C}^{r \times r})$ such that for any $(x_1, x, v) \in [-T, T] \times SM$
    \begin{equation}\label{eq:complex-transport-complex-x-ray}
        \big(\partial_{x_1} + iX(x, v) + A(x_1, x)(\partial_{x_1} + i v)\big) u(x_1, x, v) = f\big(x_1, x, \otimes^m(\partial_{x_1} + iv)\big).
    \end{equation}
    Moreover, $\mc{C}(f) \equiv 0$ if and only if $u \equiv 0$ in the component of $\partial ([-T, T] \times SM)$ of the complement of $\supp(f)$.
\end{lemma}
\begin{proof}
    By Lemma \ref{lemma:del-bar-source} (and its parametric version), for any $(x_1, x, v) \in \mathbb{R} \times M$ there exists a neighbourhood $\mc{U} \ni (x_1, x, v)$ and $u_{\mc{U}} \in C^\infty(\mc{U} \times \mathbb{C}, \mathbb{C}^{r \times r})$ so that for each $(x_1', x', v') \in \mc{U}$, $u_{\mc{U}}(x_1', x', v', \bullet)$ solves \eqref{eq:del-bar-x-ray} (one has to divide by a factor of $2^{1 - m}$ to get a solution of \eqref{eq:complex-transport-complex-x-ray} in the end which we assume from now on). Covering $\mathbb{R} \times SM_e^\circ$ by open sets and using uniqueness of solutions (given by Lemma \ref{lemma:del-bar-source}), we obtain $\widetilde{u} \in C^\infty(\mathbb{R} \times SM_e^\circ \times \mathbb{C}, \mathbb{C}^{r \times r})$ satisfying \eqref{eq:del-bar-x-ray} similarly. Now given any $(x_1, x, v) \in [-T, T] \times SM$ define
    \[
        u(x_1, x, v) := \widetilde{u}(x_1, x, v, 0),
    \]
    and we obtain $u \in C^\infty([-T, T] \times SM, \mathbb{C}^{r \times r})$ satisfying \eqref{eq:complex-transport-complex-x-ray} (similarly to the proof of Theorem \ref{thm:main-theorem-gaussian-beams}, Step 1).

    If $\mc{C}f \equiv 0$, by Lemma \ref{lemma:complex-x-ray-equivalence} we know that $u \equiv 0$ for large $|z|$ in the bicharacteristic leaves, and in fact by holomorphicity $u \equiv 0$ as in the statement of the lemma. The converse follows obviously from \eqref{eq:complex-transport-complex-x-ray}. This completes the proof.
\end{proof}

\subsection{Kernel of the complex ray transform}

Let us now discuss the kernel of $\mc{C}$. We first recall some notation about tensors; for more details see \cite[Section 2]{Cekic-Lefeuvre-20}. The \emph{symmetrised covariant derivative} is defined for each $m \in \mathbb{Z}_{\geq 0}$ as
\[
    D: C^\infty(M, \otimes_S^m T^*M) \to C^\infty(M, \otimes_S^{m+1}T^*M), \quad D := \mathrm{Sym} \circ \nabla,
\]
where $\operatorname{Sym}$ is the symmetrisation operation on tensors and $\nabla$ is the Levi-Civita covariant derivative. The \emph{pullback} operation is defined as
\[
    \pi_m^*: \otimes_S^m T_x^*M \to C^\infty(S_xM), \quad \pi_m^*T := T_x(\otimes^m v), \quad x \in M, \quad v \in S_xM,
\]
and hence it also acts as $\pi_m^*: C^\infty(M, \otimes_S^m T^*M) \to C^\infty(SM)$; note that we may extend $\pi_m^*$ naturally to $x_1$-dependent tensors and obtain $x_1$-dependent functions. Then $D$ satisfies the following fundamental relation
\begin{equation}\label{eq:fundamental-relation}
    X \pi_m^* T = \pi_{m + 1}^* D T, \quad T \in C^\infty(M, \otimes_S^m T^*M),
\end{equation}
where $X$ is the geodesic vector field on $SM$. We will write $\nabla^{e \oplus g}$ and $D^{e \oplus g}$ for the Levi-Civita covariant and symmetrised covariant derivatives on $(\mathbb{R} \times M, e\oplus g)$, respectively. Next, any $u \in C^\infty(SM)$ has an expansion into \emph{Fourier modes}, $u = \sum_{i = 0}^\infty u_i$, where the sum converges in $L^2$, and for $i = 0, 1, \dotsc$, and $x \in M$ we have $u_i(x, \bullet) \in \Omega_i(x)$ where 
\[
    \Omega_i(x) = \{f \in C^\infty(S_xM) \mid -\Delta_{S_xM} f(v) = i(i + n - 3) f(v)\},
\]
where $-\Delta_{S_xM}$ denotes the Laplacian of $(S_xM, g|_{S_xM})$ (recall here that $\dim M = n - 1$). If $u_i \not \equiv 0$ for infinitely many $i$'s, define the \emph{degree} of $u$ to be $\deg u := \infty$; otherwise, set $\deg u := \max \{i \in \mathbb{Z}_{\geq 0} \mid u_i \not \equiv 0\}$. Next, denote by $\otimes^m_S T^*M|_{0-\Tr}$ the bundle of \emph{trace-free} symmetric tensors of degree $m$; we have $T \in \otimes^m_S|_{0-\Tr} T_x^*M$ if
\[
    \sum_i T(\e_i, \e_i, \dotsc) \equiv 0,
\]
where $(\e_i)_{i = 1}^{n - 1}$ is an orthonormal basis of $T_xM$. It is well-known that
\[
    \pi_m^*: \otimes^m_S|_{0-\Tr} T_x^*M \to \Omega_m(x), \quad m \in \mathbb{Z}_{\geq 0}
\]
is an isomorphism.

The notions of degree and pullback extend naturally to $x_1$-dependent symmetric tensors that are independent of $dx_1$. Moreover, for each $T \in C^\infty([-T, T] \times M, \otimes_S^{m}T^*(\mathbb{R} \times M))$ there is a \emph{unique} decomposition
\begin{equation}\label{eq:unique-decomposition}
    T = \sum_{j = 0}^m \operatorname{Sym}(\otimes^j dx_1 \otimes T_{m - j}),\quad  T_{m - j} \in  C^\infty([-T, T] \times M, \otimes_S^{m - j}|_{0-\Tr}T^*M).
\end{equation}
We are in shape to state a few basic properties of the complex ray transform. In what follows $\mc{C}_m$ will denote the complex ray transform for $A = 0$.

\begin{proposition}\label{prop:kernel}
    For all $m \in \mathbb{Z}_{\geq 0}$, we have the following inclusion
    \begin{equation}\label{eq:inclusion}
        \{D^{e \oplus g} p \mid p \in C_{\comp}^\infty((-T, T) \times M^\circ, \otimes_S^m T^*(\mathbb{R} \times M)\} \subset \ker \mc{C}_{m + 1}.
    \end{equation}
\end{proposition}
\begin{proof}
    Write $p = \sum_{j} \operatorname{Sym}(dx_1^j \otimes p_{m - j})$ according to the decomposition \eqref{eq:unique-decomposition}. For any $(x_1, x, v) \in (-T, T) \times SM$ we compute
    \begin{align}\label{eq:complex-fundamental-relation}
    \begin{split}
        &D^{e \oplus g}p \big(\otimes^{m+ 1} (\partial_{x_1} + iv)\big)\\ 
        &= \sum_j \operatorname{Sym}\big(\otimes^{j +1} dx_1 \otimes \partial_{x_1} p_{m - j} + \otimes^j dx_1 \otimes \nabla p_{m - j}\big) \big(\otimes^{m+ 1} (\partial_{x_1} + iv)\big)\\
        &= \sum_j i^{m - j} \pi_{m - j}^* (\partial_{x_1} p_{m - j})(x_1, x, v) + \sum_j i^{m + 1 - j} \pi_{m + 1 - j}^*(D p_{m - j})(x_1, x, v)\\
        &= (\partial_{x_1} + i X) \big(\sum_j i^{m - j} \pi_{m - j}^* p_{m - j}\big)(x_1, x, v),
    \end{split}
    \end{align}
    where in the first line we used the formula $\nabla^{e \oplus g}T = dx_1 \otimes \partial_{x_1} T + \nabla T$ valid for arbitrary tensors $T$, while in the last line we used \eqref{eq:fundamental-relation}. Therefore by Lemma \ref{lemma:complex-transport-reduction} the required inclusion holds.
\end{proof}

The following extends naturally the injectivity property of the $X$-ray transform on tensors.

\begin{definition}\label{def:complex-ray-injective}
    We say that the complex ray transform $\mc{C}_{m + 1}$ is \emph{injective} if equality holds in \eqref{eq:inclusion}.
\end{definition}

\subsection{The non-linear problem: complex parallel transport}\label{ssec:complex-parallel-transport} Clearly, the complex ray transform depends linearly on $f$. We define the \emph{non-linear} counterpart of the above injectivity problem in the following. Let us first introduce the group of holomorphic invertible matrices in the complex plane
\[
    \mc{G} := \{H \in C^\infty(\mathbb{C}, \GL_{r}(\mathbb{C})) \mid \partial_{\bar{z}} H \equiv 0\}.
\]
We are now in shape to give the definition of complex parallel transport.

\begin{definition}\label{def:complex-parallel-transport-0}
    Let $A$ be a smooth connection with compact support in $(-T, T) \times M^\circ$, and let $(x, v)$ be an incoming vector generating a maximal geodesic $\gamma$ in $(M, g)$ of length $L$. Let $D := [-T, T] \times [0, L]$. By \cite[Lemma 5.8]{nakamura-uhlmann-02}, there is smooth invertible matrix function $U_A$ on $\mathbb{C}$ such that $\partial_{\bar{z}} U_A + A(\partial_{\bar{z}}) U_A = 0$. Define the \emph{complex parallel transport} at $(x, v)$
    \[
         \mc{P}_A(x, v) := U_A|_{\partial D} \mod \mc{G},
    \]
    where by modulo $\mc{G}$ we mean the matrix function $U_A|_{\partial D}$ taken modulo the action by right multiplication by elements of $\mc{G}$.
\end{definition}

We first note $\mc{P}_A(x, v)$ is well-defined: indeed any other solution $U_A'$ to $\partial_{\bar{z}} U_A' + A(\partial_{\bar{z}}) U_A' = 0$ satisfies $U_A' \equiv U_A \mod \mc{G}$. 

\begin{remark}\rm We make a few additional remarks about $\mc{P}_A(x, v)$ and relation to Cauchy data. Any solution $u$ to 
    \begin{equation}\label{eq:XYZ}
        \partial_{\bar{z}} u + A (\partial_{\bar{z}}) u = 0,\, z \in D,\quad u|_{\partial D} = F \in C^\infty(\partial D, \mathbb{C}^{r\times r})
    \end{equation} 
    has to satisfy the following equation
    \[
        \partial_{\bar{z}} w = 0, \quad w|_{\partial D} = U_A^{-1}F|_{\partial D},
    \]
    where $w := U_A^{-1}F$. Therefore the space $\mathrm{dom}(A)(x, v)$ of boundary values $F$ for which \eqref{eq:XYZ} has a (unique) solution is given by applying $U_A$ to the space of boundary values of holomorphic function on $D$. One can easily check that $\mathrm{dom}(A)(x, v)$ is independent of the choice of $U_A$. Also, it can be checked that $u$ is invertible on $D$ if and only if the (topological) degree of $\det F|_{\partial D}$ is equal to zero. 
\end{remark}

We next define what it means for two connections $A$ and $B$ to have equal parallel transports; afterwards we will see that this is the same as $\mc{P}_A(x, v) = \mc{P}_B(x, v)$.

\begin{definition}\label{def:complex-parallel-transport}
Let $A$ and $B$ be two connections as in Definition \ref{def:complex-parallel-transport-0}. Say that $A$ and $B$ have equal \emph{complex parallel transport} at $(x, v)$ if there exists $G \in C^\infty(\mathbb{C}, \mathbb{C}^{r \times r})$ such that
\begin{equation}\label{eq:complex-non-linear}
    \partial_{\bar{z}} G + A (\partial_{\bar{z}}) G(z) - G(z) B(\partial_{\bar{z}}) = 0, \quad G|_{\partial([-T, T] \times \gamma)} = \id,
\end{equation}
where $\gamma$ is the maximal geodesic in $(M, g)$ generated by $(x, v)$.
\end{definition}

It follows from the definition that if $A$ and $B$ have equal complex parallel transports, and $G$ is as in Definition \ref{def:complex-parallel-transport}, then the following relation between the complex ray transforms holds:
\[
    \mc{C}_A = \mc{C}_B G^{-1}. 
\]    
We now have:

\begin{proposition}
    Two connections $A$ and $B$ as in Definition \ref{def:complex-parallel-transport} have equal complex parallel transports at $(x, v)$ if and only if $\mc{P}_A(x, v) = \mc{P}_B(x, v)$.
\end{proposition}

\begin{proof}
   If $A$ and $B$ are have equal complex parallel transports at $(x, v)$, and $G$ is as in Definition \ref{def:complex-parallel-transport}, then a direct computation gives that we may take $U_B = G^{-1}U_A$. Since $G|_{\partial D} = \id$, we get $\mc{P}_A(x, v) \equiv U_A|_{\partial D} \equiv \mc{P}_B(x, v)  \mod \mc{G} $. 
    
   Conversely, assume $\mc{P}_A(x, v) = \mc{P}_B(x, v)$. By definition there is $H \in \mc{G}$ such which agrees with $U_B^{-1}U_A$ on $\partial D$; moreover, since $U_B^{-1}U_A$ is holomorphic outside $D$, $H$ actually agrees with $U_B^{-1}U_A$ there. We now set $G := U_A H^{-1} U_B^{-1}$, so $G \equiv \id$ outside $D$, and a direct computation shows that $G$ satisfies \eqref{eq:complex-non-linear}. This completes the proof.    
\end{proof}

\begin{remark}\label{remark}\rm
    Alternatively, we may define $\mc{P}_A(x, v) := \mathrm{dom}(A)(x, v)$ as the set of Cauchy data. Then, it is easy to see that $\mc{P}_A(x, v) = \mc{P}_B(x, v)$ is again equivalent to $A$ and $B$ having equal complex parallel transports. Still alternatively, we may define $\mc{P}_A(x, v)$ as a map between suitable subspaces of $C^\infty([-T, T] \times \{0\})$ and $C^\infty([-T, T] \times \{L\})$ defined by solutions of \eqref{eq:XYZ}. It would be interesting to compare the domains of definition of $\mc{P}_A(x, v)$ and $\mc{P}_B(x, v)$ for two different connections $A$ and $B$.
\end{remark}

Observe that the complex parallel transport is the analogue of the usual parallel transport problem for connections (see \cite{paternain-13} for instance) in the following sense. Indeed, if $P_A(t), P_B(t) \in \mathbb{C}^{r \times r}$ denote parallel transport matrices along a curve $\gamma: [0, L] \to M$ (from $\gamma(0)$ to $\gamma(t)$), then $G(t) := P_A(t) P_B(t)^{-1}$ satisfies
\begin{equation}\label{eq:real-non-linear}
    \dot{G}(t) + A(\partial_t) G(t) - G(t) B(\partial_t) = 0.
\end{equation}

If $P_A(L) = P_B(L)$ where $L$ is the length of $\gamma$, note that \eqref{eq:real-non-linear} implies that $A$ and $B$ are gauge-equivalent along $\gamma$ with an isomorphism fixing the boundary of $\gamma$. The analogous fact does not hold in dimension $\dim M = 1$ for the complex parallel transport:

\begin{proposition}\label{prop:counterexample}
    Assume $(M, g)$ is identified with a closed interval and equipped with the Euclidean metric. Then, there are examples of $A$ and $B$ as in Definition \ref{def:complex-parallel-transport} which are \emph{not} gauge equivalent but have the same complex parallel transport; we may moreover take $A$ and $B$ to be unitary.
\end{proposition}
\begin{proof}
    For simplicity, we assume $B = 0$ and give an example of a connection which is not gauge equivalent to $B$ but has the same complex parallel transport. Let $A_0$ be an arbitrary smooth connection with compact support over $(-T, T) \times M^\circ$ (on the product vector bundle with fibre $\mathbb{C}^r$). By \cite[Lemma 5.8]{nakamura-uhlmann-02}, we may solve
    \[
        \partial_{\bar{z}} G_0 + A_0(\partial_{\bar{z}}) G_0 = 0, \quad G_0 \in C^\infty(\mathbb{C}, \GL_r(\mathbb{C})).
    \]
    Let $z_0 \in \mathbb{C}$ be an arbitrary point; we may without loss of generality assume $G_0(z_0) = \id$. Let $U_0 \ni z_0$ be a small open neighbourhood where $G_0 = e^{\Psi_0}$ for some $\Psi_0 \in C^\infty(U_0, \mathbb{C}^{r \times r})$. Extend $\Psi_0$ smoothly to a compactly supported function $\Psi \in C^\infty_{\comp}((-T, T) \times M^\circ, \mathbb{C}^{r \times r})$. Set $G:= e^{\Psi}$ and choose $A$ with compact support in $(-T, T) \times M^\circ$ such that $A(\partial_{\bar{z}}) := -\partial_{\bar{z}} G \cdot G^{-1}$ and $A|_{U_0} = A_0$. It follows by definition that $A$ and $B = 0$ have the same complex parallel transport. Since $A_0$ was arbitrary this proves the first claim; the second claim follows by noting that we may simply set $A(\partial_z) := -A(\partial_{\bar{z}})^*$ to get that $A$ is additionally unitary.    
\end{proof}

We end this short paragraph by noting that \eqref{eq:complex-non-linear} still implies that $G$ is invertible (as for the usual parallel transport problem).

\begin{proposition}\label{prop:G-invertible}
    In the notation of Definition \ref{def:complex-parallel-transport}, the matrix function $G$ is invertible. 
\end{proposition}
\begin{proof}
    By Jacobi's formula we deduce from \eqref{eq:complex-non-linear} that
    \[
        \partial_{\bar{z}} \det G + \Tr(A - B)(\partial_{\bar{z}}) \det G = 0, \quad \det G|_{\mathbb{C} \setminus [-T, T] \times \gamma} \equiv 1.
    \]
    where the latter condition follows from Lemma \ref{lemma:complex-x-ray-equivalence}. By Lemma \ref{lemma:del-bar-source} there is a $\Psi \in C^\infty(\mathbb{C})$ such that
    \[
        \partial_{\bar{z}} \Psi + \Tr(A - B)(\partial_{\bar{z}}) = 0.
    \]
    It follows that $\partial_{\bar{z}} (e^{-\Psi} \det(G)) \equiv 0$; by the argument principle applied to a disk $D$ with sufficiently large radius (so that $\det G = 1$ on $\partial D$), we conclude that $\det G$ has no zeroes in $D$. This completes the proof.
\end{proof}

\subsection{Injectivity for the linear and non-linear problems} 
Recall that the $X$-ray transform on $(M, g)$ is called \emph{injective} if the transport equation on $SM$
\[
    Xu = \pi_m^*f, \quad u|_{\partial SM} = 0, \quad u \in C^\infty(SM),
\]
implies that $\deg u \leq m - 1$ and so by \eqref{eq:fundamental-relation} if $u = \pi_{m - 1}^*T$ we have $f \equiv DT$ with $T|_{\partial M} = 0$. Next, we show that for $r = 1$ the complex ray transform reduces to the usual $X$-ray transform on the transversal $(M, g)$.

\begin{proposition}\label{prop:injectivity}
    Assume $r = 1$ and $\mc{C}f \equiv 0$. Then: 
    \begin{itemize}[itemsep=5pt]
        \item[1.] Assume $A = 0$. If $m = 0$, then $f \equiv 0$; if $m = 1$, then $\exists p \in C_\comp^\infty((-T, T) \times M^\circ)$ such that $f = dp$;
        
        \item[2.] If $f = -A$, then $\exists p \in C^\infty((-T, T) \times M^\circ)$ such that $A = p^{-1} dp$, where $p - 1$ has compact support. As a consequence, if two connections $A$ and $B$ have same complex parallel transports, then they are gauge equivalent. 
    \end{itemize}
\end{proposition}
\begin{proof}
    \emph{Item 1 for $m = 0$.} Apply Lemma \ref{lemma:complex-transport-reduction} to make the reduction to a transport equation. Take the Fourier transform $\widehat{\bullet}$ in the variable $x_1$ to get
    \[
        (\xi_1 + X(x, v)) \widehat{u}(\xi_1, x, v) = -i\widehat{f}(\xi_1, x), \quad \xi_1 \in \mathbb{R}.
    \]
    For fixed $\xi_1$, this is an \emph{attenuated} ray transform: indeed, for a maximal geodesic $\gamma: [0, L] \to M$ connecting boundary points, integrating by parts this yields
    \[
        \int_0^L e^{\xi_1 t} \widehat{f}(\xi_1, \gamma(t))\,dt = 0.
    \]
    For $|\xi_1|$ small enough, this is known to give $\widehat{f}(\xi_1, x) \equiv 0$, see \cite[Theorem 7.1]{dos-santos-ferreira-keng-salo-uhlmann-09}, which by analyticity of $\widehat{f}$ in the $\xi_1$ variable gives $f \equiv 0$.
    \medskip

    \emph{Item 1 for $m = 1$.} Write $f = f_1 dx_1 + f_{\mathrm{tr}}$, where $f_{\mathrm{tr}}$ is an $x_1$ dependent $1$-form on $M$; if $(x_i)_{i = 2}^n$ are global coordinates on $M$, then $f_{\mathrm{tr}} = \sum_{i = 2}^n f_i dx_i$. Similarly to the preceding paragraph, by Lemma \ref{lemma:complex-transport-reduction} we have
    \[
        (\xi_1 + X(x, v)) \widehat{u}(\xi_1, x, v) = -i\widehat{f}_1(\xi_1, x) + \widehat{f}_{\mathrm{tr}}(\xi_1, x, v), \quad \xi_1 \in \mathbb{R},
    \]
    where $\widehat{f}_{\mathrm{tr}}(\xi_1, x, v) := \sum_{j = 2}^n \widehat{f}_j(\xi_1, x) dx_j(v)$. By injectivity of the attenuated $X$-ray transform of $1$-forms, Theorem \cite[Theorem 7.1]{dos-santos-ferreira-keng-salo-uhlmann-09}, this gives for all $|\xi_1|$ small enough that $p(\xi_1, x) :=\widehat{u}(x_1, x, v)$ does not depend on the $v$ variable; hence by analyticity of the Fourier transform $u$ does not depend on the velocity variable either. Therefore we can identify (using \eqref{eq:fundamental-relation})
    \[
        \xi_1 \widehat{u}(\xi_1, x) = -i \widehat{f}_1(\xi_1, x), \quad \pi_1^* d_x \widehat{u}(\xi_1, x) = \widehat{f}_{\mathrm{tr}}(\xi_1, x, v). 
    \]
    Taking the inverse Fourier transform we obtain $\partial_{x_1} u \equiv f_1$ and $d_xu \equiv f_{\mathrm{tr}}$, which combined is equivalent to $f \equiv du$ on $\mathbb{R} \times M$, completing the proof.
    \medskip

    \emph{Item 2.} Observe first that $A$ and $B$ have the same complex parallel transport if and only if $A - B$ has the same complex parallel transport as the trivial connection, so the second claim follows from the first one. By Lemma \ref{lemma:complex-transport-reduction}, there exists $u$ with compact support such that
    \[
        (\partial_{x_1} + iX(x, v) + A(\partial_{x_1} + iv)) (1+ u) = 0.
    \]
    Note that this is equivalent to saying that $A$ has the same parallel transport as the trivial connection, so Proposition \ref{prop:G-invertible} implies $1 + u$ is nowhere zero. Therefore $1 + u = e^{\Psi}$ in each bicharacteristic leaf for some smooth $\Psi$ which we can assume is zero for $|z|$ large. Note that $\partial_{\bar{z}} \Psi = - A(\partial_{\bar{z}})$ and by uniqueness of solutions (for instance by invoking Lemma \ref{lemma:del-bar-source}) these solutions must glue together to give $\Psi \in C^\infty(\mathbb{R} \times SM)$ such that
    \[
        (\partial_{x_1} + i X) \Psi = - A(\partial_{x_1} + iv).
    \]
     By Item 1, $\Psi$ and so in turn $1 + u$ are independent of the velocity variable. Then $p := 1 + u$ in fact satisfies the required condition.
\end{proof}

\begin{remark}\rm
    The strategy of proof of Proposition \ref{prop:injectivity}, Item 2, is very similar to \cite[Section 6.2]{dos-santos-ferreira-keng-salo-uhlmann-09}, as well as \cite[Theorem 6.1]{cekic-17}. In fact, the latter proof is also valid in the more general case where $(M, g)$ has injective $X$-ray transform on functions and $1$-tensors (with the analogous definition of $\mc{C}$ in the non-simple case). Moreover, \eqref{eq:homomorphism-connection-complex-transport} and \eqref{eq:endomorphism-potential-complex-transport} can both be interpreted as statements about a certain complex ray transform vanishing.
\end{remark}

In fact, we can also prove
\begin{proposition}\label{prop:dimension-two}
    Assume $r = 1$; if the $X$-ray transform on $m$-tensors in $(M, g)$ is injective, then the complex ray transform $\mc{C}_m$ is also injective. In particular, this holds if $\dim M = 2$.
\end{proposition}
\begin{proof}
    Assume $f$ satisfies $\mc{C}_m f = 0$; write $f = \sum_j \operatorname{Sym}(\otimes^j dx_1 \otimes f_{m - j})$ for the decomposition of $f$ as in \eqref{eq:unique-decomposition}. By Lemma \ref{lemma:complex-transport-reduction} we get a smooth compactly supported function $u$ that satisfies
    \begin{equation}\label{eq:complex-ray-tensors-transport-form-local}
        (\partial_{x_1} + iX) u = f\big(\otimes^{m} (\partial_{x_1} + iv)\big).
    \end{equation}
    Applying the Fourier transform in the $x_1$ direction, we get
    \[
        (\xi_1 + X) \widehat{u} = \sum_{j} i^{m - j} \widehat{\pi_{m - j}^* f_{m - j}}(\xi_1, x, v).
    \]
    Now using the injectivity of the $X$-ray transform on $(M, g)$ for $\xi_1 = 0$ we get $\deg \widehat{u}(0, x, v) \leq m - 1$, and inductively by differentiating the equation above in $\xi_1$, we get that 
    \begin{equation}\label{eq:bla-bla}
        \deg \partial_{\xi_1}^k \widehat{u}(0, x, v) \leq m - 1, \quad k \in \mathbb{N}_0,
    \end{equation}
    see \cite[Proof of Theorem 6.1]{cekic-17} for more details. Using the analytic expansion at zero, this implies that $\deg \widehat{u}(\xi_1, x, v) \leq m - 1$ for small $|\xi_1|$ and by using analyticity again for all $\xi_1$. More precisely, for the latter claim we write $u = \sum_i u_i$ as an expansion in Fourier modes, $\deg u_i = i$, and notice that $\widehat{u} = \sum_i \widehat{u}_i$ is the expansion into Fourier modes of $\widehat{u}$ as the Fourier transform commutes with vertical derivatives; also, each $\widehat{u}_i$ is analytic in $\xi_1$. For $i \geq m$, since $\widehat{u}_{i} \equiv 0$ for $|\xi_1|$ small enough by \eqref{eq:bla-bla}, by analyticity we have $\widehat{u}_i \equiv 0$ and the claim follows. In particular, we get that $\deg u \leq m - 1$. 

    Now write 
    \[
        u = \sum_{j = 0}^{m - 1} i^{m - 1 - j} \pi_{m - 1 - j} ^* p_{m - 1 - j}
    \]
    for some $x_1$-dependent (but $dx_1$-independent), trace-free, symmetric $(m - 1 - j)$-tensors $p_{m - 1 - j}$. Setting $p := \sum_j \operatorname{Sym}(\otimes^j dx_1 \otimes p_{m - j - 1})$, the computation in \eqref{eq:complex-fundamental-relation} shows that
    \[
        D^{e \oplus g}p \big(\otimes^{m+ 1} (\partial_{x_1} + iv)\big) = (\partial_{x_1} + i X) u = f \big(\otimes^{m+ 1} (\partial_{x_1} + iv)\big),
    \]
    where in the last equality we used \eqref{eq:complex-ray-tensors-transport-form-local}. By the uniqueness of the decomposition \eqref{eq:unique-decomposition}, we must have $D^{e \oplus g}p \equiv f$ which completes the proof of injectivity of $\mc{C}_m$.
    
    The statement about dimension two is a consequence of \cite{Paternain-Salo-Uhlmann-13}.
\end{proof}

\bibliographystyle{alpha}
\bibliography{Biblio}

\end{document}